\documentclass[11pt]{article}
\usepackage{amsmath,amssymb,amsthm}
\usepackage{color}
\usepackage{bm}
\usepackage[ruled,vlined]{algorithm2e}
\usepackage{algorithmic}
\usepackage[normalem]{ulem}
\usepackage{euscript}
\usepackage{mathrsfs}  
\usepackage{graphicx}
\usepackage{enumitem}
\usepackage{float}
\usepackage{placeins}
\usepackage{subfig}
\newtheorem{definition}{Definition}[section]

\newtheorem{proposition}[definition]{Proposition}

\renewcommand{\max}{\textnormal{max}}

\newcommand{\R}{\mathbb{R}}

\newcommand{\vPi}{\boldsymbol{\Pi}}
\newcommand{\vepsilon}{\boldsymbol{\epsilon}}

\newcommand{\vmu}{\boldsymbol{\mu}}
\newcommand{\veta}{\boldsymbol{\eta}}

\newcommand{\vGamma}{\boldsymbol{\Gamma}}

\newcommand{\vtheta}{\boldsymbol{\theta}}
\newcommand{\vTheta}{\boldsymbol{\Theta}}
\newcommand{\vPsi}{\boldsymbol{\Psi}}
\newcommand{\vomega}{\boldsymbol{\omega}}
\newcommand{\vOmega}{\boldsymbol{\Omega}}

\newcommand{\vSigma}{\boldsymbol{\Sigma}}

\newcommand{\vb}{\mathbf{b}}

\newcommand{\vd}{\mathbf{d}}
\newcommand{\ve}{\mathbf{e}}

\newcommand{\vr}{\mathbf{r}}
\newcommand{\vs}{\mathbf{s}}

\newcommand{\vu}{\mathbf{u}}
\newcommand{\vv}{\mathbf{v}}
\newcommand{\vw}{\mathbf{w}}
\newcommand{\vx}{\mathbf{x}}
\newcommand{\vy}{\mathbf{y}}

\newcommand{\vzero}{\mathbf{0}}
\newcommand{\vA}{\mathbf{A}}
\newcommand{\vB}{\mathbf{B}}

\newcommand{\vH}{\mathbf{H}}
\newcommand{\vI}{\mathbf{I}}
\newcommand{\vK}{\mathbf{K}}

\newcommand{\vM}{\mathbf{M}}
\newcommand{\vN}{\mathbf{N}}
\newcommand{\vP}{\mathbf{P}}
\newcommand{\vQ}{\mathbf{Q}}
\newcommand{\vR}{\mathbf{R}}

\newcommand{\vT}{\mathbf{T}}
\newcommand{\vU}{\mathbf{U}}
\newcommand{\vV}{\mathbf{V}}

\newcommand{\vX}{\mathbf{X}}
\newcommand{\vY}{\mathbf{Y}}
\newcommand{\vZ}{\mathbf{Z}}

\newcommand{\Id}{\mathbf{I}}

\newcommand{\calF}{\mathcal{F}}

\newcommand{\calH}{\mathcal{H}}
\newcommand{\calK}{\mathcal{K}}
\newcommand{\calM}{\mathcal{M}}
\newcommand{\calN}{\mathcal{N}}
\newcommand{\calT}{\mathcal{T}}

\newcommand{\calX}{\mathcal{X}}

\newcommand{\rank}{\mathsf{rank}}

\newcommand{\trace}{\mathsf{tr}}
\newcommand{\diag}{\textnormal{diag}}

\newcommand{\Argmin}{\textrm{Argmin}}
\newcommand{\argmin}{\textrm{argmin}}

\newcommand{\expect}{\mathbb{E}}
\newcommand{\logdet}{\mathsf{logdet}}
\renewcommand{\det}{\mathsf{det}}
\newcommand{\lrp}[1]{\left( #1  \right)}
\newcommand{\lrb}[1]{\left[ #1  \right]}

\newcommand{\inn}[2]{\left\langle #1 , #2 \right\rangle}
\newcommand{\norm}[1]{\left\| #1  \right\|}

\newcommand{\vertiii}[1]{{\left\vert\kern-0.25ex\left\vert\kern-0.25ex\left\vert #1 \right\vert\kern-0.25ex\right\vert\kern-0.25ex\right\vert}}

\newcommand{\mc}[1]{\mathcal{#1}}

\usepackage[textsize=tiny]{todonotes}
\usepackage[margin=1.1in]{geometry}
\begin{document}

\title{Efficient iterative methods for hyperparameter estimation in large-scale linear inverse problems}
\author{ Khalil A Hall-Hooper, Arvind K Saibaba, Julianne Chung, and
Scot M Miller}
\maketitle

\abstract{We study Bayesian methods for large-scale linear inverse problems, focusing on the challenging task of hyperparameter estimation. Typical hierarchical Bayesian formulations that follow a Markov Chain Monte Carlo approach are possible for small problems with very few hyperparameters but are not computationally feasible for problems with a very large number of unknown parameters. In this work, we describe an empirical Bayes (EB) method to estimate hyperparameters that maximize the marginal posterior, i.e., the probability density of the hyperparameters conditioned on the data and then we use the estimated values to compute the posterior of the inverse parameters. For problems where the computation of the square root and inverse of prior covariance matrices are not feasible, we describe an approach based on the generalized Golub-Kahan bidiagonalization to approximate the marginal posterior and seek hyperparameters that minimize the approximate marginal posterior. Numerical results from seismic and atmospheric tomography demonstrate the accuracy, robustness, and potential benefits of the proposed approach.}

\maketitle

\section{Introduction}
\label{sec::introduction}

Inverse problems arise in many important applications, where the main goal is to infer some unknown parameters from given observed data. Bayesian approaches provide a robust framework for tackling such problems, where the posterior distribution of the unknown parameters combines information provided by the observations (e.g., the likelihood function) and prior knowledge about the unknown parameters (e.g., the prior).  However, for problems where the prior and/or the likelihood contain unspecified hyperparameters, \textit{hierarchical} Bayesian approaches can be used such that all parameters (including the inverse parameters and the hyperparameters) can be inferred from the data in a ``fully Bayesian'' framework.

There are several challenges, especially for large-scale inverse problems where the number of unknowns and the size of the observed data sets are large. One of the key challenges is model specification, whereby appropriate prior distributions (including hyperpriors) and hyperparameters for each level of the hierarchy must be specified. It is important to select suitable prior distributions that capture the characteristics of the problem accurately, but finding informative and realistic priors becomes more difficult as the complexity and size of the problem increase.  Another key challenge of a fully Bayesian framework is that even for linear-Gaussian inverse problems, where the forward model is linear and both the prior and error distributions are Gaussian, the posterior distribution may no longer be Gaussian, making sampling approaches significantly more challenging.  These challenges are compounded by the need to perform multiple inference runs for multiple parameters, each of which requires an accurate and efficient solver for the inverse problem.

To circumvent some of the challenges of fully Bayesian approaches, empirical Bayes (EB) approaches can be used that first estimate the hyperparameters by maximizing a marginal likelihood function, and then use the estimated values to compute the posterior of the inversion parameters.  However, optimization for hyperparameters remains a challenging and expensive task (e.g., due to multiple costly objective function and gradient evaluations, each of which requires accurate inference), and much less attention in the literature has been given to estimating hyperparameters that govern prior distributions.

We focus on large-scale linear-Gaussian inverse problems, where the forward and adjoint models as well as the prior covariance matrix are only accessible via matrix-vector multiplications.  This circumstance may also arise from covariance kernels that are defined on irregular grids or come from a dictionary collection. Covariance kernels from the Mat{\' e}rn class will be considered, although various classes of parameterized covariance kernels could be considered in this framework. For such scenarios, generalized Golub-Kahan (genGK) iterative methods have been proposed for inference \cite{AKS_Chung_genHyBR}, but the estimation of the hyperparameters is much more challenging. 

\paragraph{Overview of Main Contributions.} In this paper, we exploit Krylov subspace projections for hyperparameter estimation using empirical Bayes methods. The main contributions are as follows:
\begin{enumerate}
\item We propose an approximation to the marginal posterior objective function and its gradient using the genGK bidiagonalization. The genGK approach only requires matrix-vector products (matvecs) involving the forward and adjoint operators, as well as with the prior covariance matrix.  However, it doesn't require applying square roots and inverses.  
\item We provide a detailed analysis of the computational costs and the resulting error for the approximations.  These approximations and theoretical bounds can be used for other tasks related to Bayesian inverse problems.
\item We develop methods for monitoring the accuracy of the resulting approximations and analysis for trace estimators for nonsymmetric matrices.
\item While our proposed approach is general, we show how to significantly accelerate the estimation of hyperparameters when the noise and prior variance are assumed to be unknown. This special case is important in many applications.
\end{enumerate}

\paragraph{Other applications of our approach.}

{Although we focus on the empirical Bayes approach, there are numerous scenarios in large-scale Bayesian inverse problems where the approximations described in this work can be used to make computational advances.  For example,}

\begin{enumerate}
\item \textbf{Fully-Bayesian approach}: 
Our approach provides efficient ways to approximate the objective function (obtained by the negative log-likelihood of the marginal posterior distribution) and its gradient using the genGK approaches. However, it is easy to see that the approaches can be easily extended to approximate the likelihood $\pi(\boldsymbol{\theta}|\vb)$. In particular, these approximations can be used to significantly accelerate Markov Chain Monte Carlo (MCMC) approaches for exploring the marginal posterior distribution and providing uncertainty estimates for the hyperparameters. 

\item \textbf{Variational Bayes approach}: In this approach, we approximate the marginal posterior distribution by minimizing the distance to another class of distributions (e.g., Gaussian). As with the empirical Bayes approach, the optimization techniques also require repeated evaluations of the objective function and its gradient. Therefore, the computational techniques developed here are also applicable to that setting.
\item \textbf{Information-theoretic approaches}: An alternative approach to estimating the hyperparameters is to optimize the information gain from the prior distribution to the posterior. One such metric is the expected Kullback-Liebler divergence, also known as the D-optimal design criterion, which has a similar form as our objective function (in fact, it only involves the log-determinant term). Therefore, the approximation techniques developed here could also be applied to estimate hyperparameters using information theoretic criteria. 
\end{enumerate}

\subsection{Related work}
Hyperparameter estimation in hierarchical Bayesian problems is an important problem that has received some attention. In the geostatistical community, an approach for estimating hyperparameters is the maximum likelihood and restricted maximum likelihood approaches~\cite{kitanidis1983geostatistical,kitanidis1995quasi,michalak2005maximum}. However, these approaches are not scalable to large-scale problems and are different from ours since we use the MAP estimate of the marginalized posterior.

In contrast to optimization-based techniques, sampling methods based on MCMC have been used to estimate the uncertainty in the hyperparameters (see, e.g.,~\cite{bardsley2018computational}). However, this approach is considerably more expensive since it requires computing the likelihood thousands of times to get reliable samples from the posterior distribution.  Another approach to estimating hyperparameters is using the semivariogram technique~\cite{brown2020semivariogram}, but this is only applicable to estimating the prior hyperparameters.

Our work is close in spirit to~\cite{ApproxEB_1} who also use a low-rank approximation to accelerate the computation of the objective function based on a marginal likelihood; to find the optimal parameters, they evaluated it over a finite grid of hyperparameters. Our approach instead uses the marginal posterior to construct the objective function and uses the genGK iterative method to approximate it. Furthermore, we show how to approximate the derivatives using the genGK approach and use it in a continuous optimization approach. 

In previous work, we developed a solver called genHyBR, also based on genGK, for Bayesian inverse problems~\cite{AKS_Chung_genHyBR} but it was only capable of estimating a limited set of hyperparameters (associated with the prior). The authors in~\cite{majumder2022kryging} also use the genGK approach in Gaussian processes, similar to ours, to estimate hyperparameters but use the profile likelihood which is not a Bayesian approach like ours. 

\subsection{Outline of the paper}
The paper is organized as follows. In Section \ref{sec::notation_problemSetup}, we provide a brief overview of hyperparameter estimation, from a hierarchical Bayes formulation to an empirical Bayes approach. We describe the generalized Golub-Kahan bidiagonalization process, which will be used in Section \ref{sec:estimation}.  That is, we use elements from the genGK bidiagonalization to approximate the EB objective function and gradient for optimization and provide an error analysis for approximations that can be used to monitor the accuracy of approximations. 
We demonstrate the performance of our approach on model inverse problems in heat conduction, seismic and atmospheric tomography in Section \ref{sec:numerics}, and conclusions and future work are provided in Section \ref{sec:conclusions}. Theoretical derivations have been relegated to the Appendix.

\section{Background and Problem Setup}
\label{sec::notation_problemSetup}

\subsection{Notation}
Let $m$ and $n$ be strictly positive integers, and denote the nonnegative reals by $\R_{+}.$ In the following, $\R^{n}$ is the standard $n$-dimensional Euclidean space with scalar product $(\vx, \vy) \mapsto \langle \vx, \vy \rangle_{2} = \vy^{\top} \vx,$ and $\calM_{m,n}(\R)$ is the normed space of all real-valued $m$-by-$n$ matrices with scalar product $(\vA, \vB) \mapsto \langle \vA, \vB \rangle_{F} = \trace \left( \vB^{\top} \vA \right),$ where $\top$ denotes the matrix transpose. We write $\calM_{m,m}(\R) \equiv \calM_{m}(\R).$ Let $\vA \in \calM_{n}(\R),$ and let $\vx$ and $\vy$ be in $\R^{n}.$ If $\vA$ is symmetric positive definite {(SPD)}, we write $\langle \vx, \vy \rangle_{\vA} \equiv \langle \vA \vx, \vy \rangle_{2}$ and $\| \vx \|_{\vA} \equiv \sqrt{\langle \vA \vx, \vx \rangle_{2}},$ with $\| \vx \|_{2} \equiv \sqrt{\langle \vx, \vx \rangle_{2}}.$ 
We also write $\norm{\vA}_{F} \equiv \sqrt{\langle \vA, \vA \rangle_{F}}$ to denote the Frobenius norm and $\|\vA\|_2$ denotes the spectral or 2-norm.

Given any realization, say $\delta,$ of a random variable, we write $X_\delta$ to denote said variable. We also use the notation $\pi[\cdot]$ to denote an arbitrary probability density.

Let $\calX$ be a real vector space, $f : \calX \rightarrow \R,$ and let $x \in \calX.$ Denote by $\Argmin \, f \equiv \left\{ x \in \calX \, | \, f(x) = \min_{z \in \calX} f(z) \right\}.$ If $\Argmin \, f$ is single-valued, then we write that unique element as $\argmin_{x \in \calX} \, f.$

\subsection{Hierarchical Bayes formulation} Let $m,$ $n,$ and $K$ be strictly positive integers with $m \leq n$. We consider a linear inverse problem of recovering an unknown vector $\vs \in \R^{n}$ of signals from a noisy vector $\vd \in \R^{m}$ of data observations. The signals and observations are characterized by the model
\begin{eqnarray}
    \vd = \vA \vs + \veta, \label{gen_inv_prob}
\end{eqnarray}
where $\vA \in \R^{m \times n}$ represents the the forward map and $\veta \in \R^{m}$ is a vector of realizations of a random variable $X_{\veta} \sim \calN(\vzero, \vR(\cdot)).$ In this setting, we assume that $\vR : (0,+\infty)^{K} \, \rightarrow \R^{m \times m},$ where $\vR(\cdot)$ is {SPD}, is a diagonal matrix, so it is easy to invert and compute the square root. 

Let $\vtheta = (\theta_{i})_{1 \leq i \leq K} \in (0,+\infty)^{K}$ be a vector of unknown hyperparameters, with prior density $\pi[\vtheta]$. Following a hierarchical Bayesian approach, we model $\vs \, | \, \vtheta$ as a realization of a Gaussian random variable $X_{\vs \, | \, \vtheta} \sim \calN(\vmu(\vtheta), \vQ(\vtheta)),$ where $\vmu : (0,+\infty)^{K} \, \rightarrow \R^n,$ $\vQ : (0,+\infty)^{K} \, \rightarrow \R^{n \times n},$ and $\vQ(\vtheta)$ is SPD, but is computationally infeasible to compute explicitly; therefore, we do not assume that we have access to its square root or its inverse. We do, however, assume that the matvecs with $\vQ(\vtheta)$ can be performed easily in $\mathcal{O}(n \log \, n)$ time by using the fast Fourier Transform (FFT), if the solution is represented on a uniform equispaced grid~\cite{NowakTenkleveCirpka}, or by using hierarchical matrices \cite{SaibabKitan_EffMethodsLS}; furthermore, $\mc{O}(n)$ time can be achieved in certain circumstances with the aid of $\calH^2$-matrices or the fast multipole method \cite{Ambikasaran_LSLinInversion}. 

Using Bayes' theorem, the posterior density, $\pi[\vs, \vtheta \, | \, \vd],$ is characterized by
\[ \pi[\vs, \vtheta \, | \, \vd] = \frac{\pi[\vd \, | \, \vs, \vtheta] \pi [\vs \, | \, \vtheta] \pi [\vtheta]}{\pi[\vd]} .\]
 In our setting, the posterior is explicitly 
\[ \pi[\vs, \vtheta \, | \, \vd] \propto \frac{\pi[\vtheta] \exp \left( - \frac{1}{2} \| \vd - \vA \vs \|^{2}_{\vR^{-1}(\vtheta)} - \frac{1}{2} \| \vs - \vmu(\vtheta) \|^{2}_{\vQ^{-1}(\vtheta)} \right)}{\det(\vR(\vtheta))^{1/2} \det(\vQ(\vtheta))^{1/2}}, \]
where  $\propto$ denotes the proportionality relation. The marginal posterior density is obtained by integrating out the unknowns $\vs$ and can be represented as
\begin{equation}\label{eqn:marginalpost}
    \pi[\vtheta \, | \, \vd] \propto \pi[\vtheta] \det(\vZ(\vtheta))^{-1/2} \exp \left( - \frac{1}{2} \| \vA \vmu(\vtheta) - \vd \|^{2}_{\vZ^{-1}(\vtheta)} \right), 
\end{equation}
where $\vZ : (0,+\infty)^{K} \rightarrow \R^{m \times m}$ has the representation $\vZ(\vtheta) = \vA \vQ(\vtheta) \vA^{\top} + \vR(\vtheta)$ \cite{WellPosed_StochExt_LP,LinIV_GenRV}.

For the case where $\vtheta$ is a \textit{fixed}, known value, or in the non-hierarchical setting, the maximum a posteriori estimate $\vs_{\rm post},$ which is obtained by minimizing the negative logarithm of the posterior distribution, i.e., $\pi_{\rm post}[\vs \,|\,\vd] \propto \pi[\vs\,|\,\vd] \pi[\vs]$ (see, e.g.,~\cite{IntroToBayesComp_CalvSomer}) and has the explicit expression
\begin{equation}\begin{aligned}
    \vs_{\rm post} & = \argmin_{\vs \in \R^n} (- \log \, \pi_{\rm post}[\vs \,|\,\vd] ) \nonumber \\
     &=  \argmin_{\vs \in \R^n} \frac{1}{2} \| \vd - \vA \vs \|^{2}_{\vR^{-1}} + \frac{1}{2} \| \vs - \vmu \|^{2}_{\vQ^{-1}},
    \label{eq:Tikhonov}
\end{aligned}
\end{equation}where we have suppressed the arguments containing $\vtheta$ to indicate that these quantities are fixed.
The MAP estimate $\vs_{\rm post}$ can equivalently be expressed in closed form as 
\begin{eqnarray}\label{eqn:map_s}
     {\vs_{\rm post}} = \vGamma_{\rm post}(\vA^{\top} \vR^{-1} \vd + \vQ^{-1} \vmu),
\end{eqnarray}
where $\vGamma_{\rm post} = (\vA^{\top} \vR^{-1} \vA + \vQ^{-1})^{-1}.$
For large inverse problems, solving \eqref{eqn:map_s} is not computationally feasible in practice, and several studies provide alternative, iterative approaches.

\subsection{The empirical Bayes method for hyperparameter estimation}

The empirical Bayes framework (also known as evidence approximation in the machine learning literature) allows one to estimate the values of the hyperparameters via the marginal posterior. {In doing so, the choice of the hyperparameters is informed by the model, the data, and any assumptions about the likelihood and prior formulations. The main idea behind the EB method~\cite{reich2019bayesian} is to estimate the hyperparameters $\vtheta_{\rm EB}$ and fix the hyperparameters in the posterior distribution. That is, we set $\pi[\vs \, | \, \vd] \approx \pi[\vs, \vtheta_{\rm EB} \, | \, \vd]$. The EB approach has some known drawbacks: first, it uses the data twice, once to estimate the hyperparameters and second to determine the unknowns $\vs$, second, it ignores the uncertainty in the hyperparameters and, therefore, underestimates the overall uncertainty. Nevertheless, it remains a popular approach in statistics especially for computationally challenging problems and we adopt this approach in the present paper.

Specifically, we choose the hyperparameters $\vtheta_{\rm EB}$ obtained from the maximum a posteriori (MAP) estimate of the marginal posterior distribution~\eqref{eqn:marginalpost}; alternatively,  it can obtained by minimizing the 
objective function $\calF (\vtheta),$ defined by
\begin{eqnarray}
\label{eq:objfunc}
    \calF (\vtheta) = - \log \, \pi [\vtheta] + \frac{1}{2} \logdet(\vZ(\vtheta)) + \frac{1}{2} \| \vA \vmu(\vtheta) - \vd \|^{2}_{\vZ^{-1}(\vtheta)}.
\end{eqnarray}
Once the optimal hyperparameters $\vtheta_{\rm EB}$ have been estimated, we can obtain the MAP estimate for $\vs$ by~\eqref{eqn:map_s}.

Anticipating that we will use a gradient-based approach for optimizing~\eqref{eq:objfunc}, we provide an analytical expression for the gradient $\nabla \calF = (\frac{\partial \calF}{\partial \theta_i})_{1 \leq i \leq K},$ 

\begin{align}
    (\nabla \calF (\vtheta))_{i} = &  - \frac{1}{\pi[\vtheta]} \frac{\partial \pi[\vtheta]}{\partial \theta_{i}} + \frac{1}{2} \left\langle \vZ^{-1}(\vtheta), \frac{\partial \vZ(\vtheta)}{\partial \theta_i} \right\rangle_{F} \nonumber \\
    & \hspace{.4cm} - \frac{1}{2} \left[ \vZ^{-1}(\vtheta)(\vA \vmu(\vtheta) - \vd) \right]^{\top} \left[ \frac{\partial \vZ(\vtheta)}{\partial \theta_i} \vZ^{-1}(\vtheta)(\vA \vmu(\vtheta) - \vd) - \vA \frac{\partial \vmu(\vtheta)}{\partial \theta_i} \right], \label{eq:gradient} 
\end{align}

Notice that there are various challenges for optimization. First, the objective function ($\calF$) is nonconvex, with potentially many local minima so this causes problems for the optimization algorithms which can be sensitive to the choice of the initial guess. Second, computing the objective function and the gradient ($\nabla \calF$) for each new candidate set of hyperparameters $\vtheta$ requires recalculating the log determinant and the inverse of an $m \times m$ matrix; even forming this matrix explicitly is expensive and should be avoided.  This can become computationally prohibitive for problems with large observational datasets. We focus our attention in this paper on the second challenge and in Section~\ref{sec:estimation} propose efficient methods to compute the objective function and the gradient.

\subsection{Generalized Golub-Kahan bidiagonalization}
The generalized Golub-Kahan (genGK) bidiagonalization was used for efficiently computing Tikhonov regularized solutions \eqref{eq:Tikhonov} in \cite{AKS_Chung_genHyBR} and for inverse UQ in \cite{AKS_Chung_Petroske_EffKrylov}.  However, in this work, we extend the use of the genGK bidiagonalization process for efficient hyperparameter estimation following an EB approach.  For completeness, we provide an overview of the genGK process.

Throughout this section, suppose we have some fixed hyperparameters $\bar{\vtheta} \in \R^{K}_{+}.$ 
Let $\vR \equiv \vR(\bar{\vtheta}),$ $\vQ \equiv \vQ(\bar{\vtheta}),$ and $\vmu = \vmu(\bar{\vtheta}).$ Given the matrices $\vA,$ $\vR,$ $\vQ$ and the vector $\vd$ (as defined in (\ref{gen_inv_prob})), with the initialization 
\[  
    \beta_1\vu_{1} =  {\vd - \vA \vmu} \qquad \alpha_1 \vv_1 = \vA^{\top} \vR^{-1} \vu_{1} 
    \]
the $j$th iteration of the genGK (lower) {bi}diagonalization process generates vectors $\vu_{j+1}$ and $\vv_{j+1}$ such that
\[ \begin{aligned}
    \beta_{j+1} \vu_{j+1} &= \vA \vQ \vv_{j} - \alpha_{j} \vu_{j}; \\
    \alpha_{j+1} \vv_{j+1} &= \vA^{\top} \vR^{-1} \vu_{j+1} - \beta_{j+1} \vv_{j},
\end{aligned} \]
where, for some natural number $i,$ the scalars $(\alpha_{i}, \beta_{i}) \in \R^{2}_{+}$ are chosen such that $\| \vu_{i} \|_{\vR^{-1}} = \| \vv_i \|_{\vQ} = 1.$
At the end of $k$ iterations, we have a lower bidiagonal matrix
\[ \vB_{k} = \begin{bmatrix} \alpha_1 &  &  &  \\ \beta_{2} & \alpha_2 & & & \\ & \beta_3 & \ddots & \\ & & \ddots & \alpha_k \\ & & & \beta_{k+1} \end{bmatrix} \in \R^{(k+1) \times k}, \]
and two matrices $\vU_{k+1} = [\vu_1 \, | \, \cdots \, | \, \vu_{k+1}] \in \R^{m \times (k+1)}$ and $\vV_{k} = [\vv_1 \, | \, \vv_2 \, | \, \cdots \, | \, \vv_{k}] \in \R^{n \times k}$ that satisfy the orthogonality relations
\begin{equation}\label{orthog_relations_1}
    \vU^{\top}_{k+1} \vR^{-1} \vU_{k+1} = \Id_{k+1} \qquad \text{and} \qquad \vV^{\top}_{k} \vQ \vV_{k} = \Id_{k}.
\end{equation}

Furthermore, these matrices satisfy the genGK relations 
\begin{align}\label{genGK_relations}
    \begin{cases}
        \vU_{k+1} \beta_1 \ve_{1} &= \vd - \vA \vmu \\ 
        \vA \vQ \vV_{k} &= \vU_{k+1} \vB_{k} \\ 
        \vA^{\top} \vR^{-1} \vU_{k+1} &= \vV_{k} \vB^{\top}_{k} + \alpha_{k+1} \vv_{k+1} \ve^{\top}_{k+1}. 
    \end{cases}
\end{align}
In floating point arithmetic, these relations are typically accurate to machine precision. However, the matrices $\vU_k$ and $\vV_k$ tend to lose orthogonality with respect to the appropriate inner products; therefore, in practice, we use complete reorthogonalization to mitigate this loss in accuracy.  
{The above process is summarized in Algorithm \ref{algorithm:genGK}.} 

\begin{algorithm}[!ht]
    \caption{genGK bidiagonalization. Call as $[\vU_{k+1}, \vV_{k+1}, \vB_{k}] = \textrm{genGK}(\vA, \vR, \vQ, \vmu, \vd, k)$.}
    \label{algorithm:genGK}
    
    \begin{algorithmic}[1]
    \REQUIRE Matrices $\vA,$ $\vR,$ and $\vQ;$ vectors $\vmu$ and $\vd.$
    \STATE 
    {$\beta_1 = \| \vd - \vA \vmu \|_{\vR^{-1}};$}
    \STATE 
    {$\vu_1 = \frac{\vd - \vA \vmu}{\beta_1};$}
    \STATE
    {$\alpha_{1} = \| \vA^{\top} \vR^{-1} \vu_1 \|_{\vQ};$}
    \STATE
    {$\vv_1 = \frac{\vA^{\top} \vR^{-1} \vu_1}{\alpha_1};$}
    \FOR{$j = 1, 2, \dots, k$} 
        \STATE
        {$\beta_{j+1} = \| \vA \vQ \vv_j - \alpha_j \vu_j \|_{\vR^{-1}};$}
        \STATE 
        {$\vu_{j+1} = \frac{\vA \vQ \vv_{j} - \alpha_{j} \vu_{j}}{\beta_{j+1}};$}
        \STATE
        {$\alpha_{j+1} = \| \vA^{\top} \vR^{-1} \vu_{j+1} - \beta_{j+1} \vv_j \|_{\vQ};$}
        \STATE
        {$\vv_{j+1} = \frac{\vA^{\top} \vR^{-1} \vu_{j+1} - \beta_{j+1} \vv_i}{\alpha_{j+1}};$}
    \ENDFOR
    \RETURN Matrices $\vU_{k+1},\vV_{k+1}, \vB_k$.
    \end{algorithmic}
\end{algorithm}

\paragraph{Low-rank approximation.}
We can reinterpret the genGK relations through the lens of projectors which leads us to a low-rank approximation of $\vA$. The low-rank approximation of $\vA$ will be critical in developing efficient approaches in Section~\ref{sec:estimation}. Define the following projectors:
\begin{align}
    \vP_{\vV_{k}} = \vV_{k} \vV^{\top}_{k} \vQ; ~~~~ \vP_{\vU_{k+1}} = \vU_{k+1} \vU^{\top}_{k+1} \vR^{-1}.
\end{align}
Then orthogonality relations in~\eqref{orthog_relations_1} imply that $\vP_{\vV_{k}}$ and $\vP_{\vU_{k+1}}$ are oblique projectors. Using this insight, we can rewrite (\ref{genGK_relations}) as follows:
\begin{align}\label{genGK_relations_transformed}
    \begin{cases}
        \vU_{k+1} \beta_1 \ve_{1} &= \vP_{\vU_{k+1}} (\vd - \vA \vmu) \\ 
        \vA \vP^{\top}_{\vV_{k}} &= \vU_{k+1} \vB_{k} \vV^{\top}_{k} \\ 
        \vA^{\top} \vP_{\vU_{k+1}} &= \vV_{k+1} \begin{bmatrix} \vB^{\top}_{k} \\ \alpha_{k+1} \ve^{\top}_{k+1} \end{bmatrix} \vU^{\top}_{k+1}. 
    \end{cases}
\end{align}
This reinterpretation of the genGK relations in (\ref{genGK_relations}) using projectors suggests a natural low-rank approximation for $\vA$:
\begin{equation}\label{approx_of_A}
    \vA \approx \tilde{\vA}_{k} = \vA \vP^{\top}_{\vV_{k}} = \vU_{k+1} \vB_{k} \vV^{\top}_{k}.
\end{equation}
{In the next section,} we will use this low-rank approximation to approximate {the objective function \eqref{eq:objfunc} and the gradient \eqref{eq:gradient} for EB hyperparameter estimation. Moreover,} the projection-based viewpoint of the low-rank approximation will be useful {for} error analysis.

\section{Efficient hyperparameter estimation}
\label{sec:estimation}
{The main goal of this section is to solve the optimization problem involving~\eqref{eq:objfunc} by exploiting the genGK bidiagonalization.  Consider minimizing the marginal posterior, 
\begin{equation}
\label{eqn:EBmin}
\min_{\vtheta \in \mathbb{R}_+^K} \calF(\vtheta)
\end{equation}
where $\calF(\vtheta)$ is defined in \eqref{eq:objfunc}. First, we use the genGK bidiagonalization to approximate the objective function and gradient, so that evaluations during optimization can be done efficiently.  Then, we provide an error analysis to quantify the errors between $\calF$ and its approximations, which will be used for monitoring the convergence of the genGK process.}

\subsection{Approximations using genGK}\label{ssec:gengk}
As before, we assume that we have a fixed $\bar{\vtheta},$ and for ease of exposition, we suppress the dependence on $\bar{\vtheta}$.  For any differentiable mapping $T: \R^{K}_{+} \rightarrow \R^{m\times n},$ we write for brevity
\[\begin{aligned}
     \partial_{\theta_i} T  = \frac{\partial T (\vtheta)}{\partial \theta_{i}} \bigg|_{\vtheta = \bar{\vtheta}} \qquad 1 \leq i \leq K.
\end{aligned}
\]
The objective function $\bar{\calF} \equiv \calF(\bar\vtheta)$ takes the form
\begin{equation}\label{calF_bar}
    \bar{\calF} \equiv - \log \, \pi  + \frac{1}{2} \logdet(\vZ) + \frac{1}{2} \| \vA \vmu - \vd \|_{\vZ^{-1}}^2,
\end{equation}
and the gradient ${\nabla \calF} = ({\partial_{\theta_i} \calF})_{1 \leq i \leq K} \in \R^{K}$ takes the form (for $1 \leq i \leq k$)
\[\begin{aligned}
     \partial_{\theta_i}\calF  = &- \frac{1}{\pi[\bar{\vtheta}]} \frac{\partial \pi[\bar{\vtheta}]}{\partial \theta_{i}} + \frac{1}{2} \left\langle \vZ^{-1}, {\partial_{\theta_i}\vZ(\bar{\vtheta})} \right\rangle_{F} \nonumber \\
    &- \frac{1}{2} \left\langle {(\partial_{\theta_i} \vZ}) \vZ^{-1}(\vA \vmu - \vd) - \vA {\partial_{\theta_i} \vmu(\bar{\vtheta})}, \vZ^{-1}(\vA \vmu - \vd)  \right\rangle_{2}.  
\end{aligned}
\]

\paragraph{Approximation to the objective function.} We can approximate $\bar{\calF}$ by substituting our low-rank approximation $\tilde{\vA}_k$ for $\vA$ in~\eqref{approx_of_A} (with the exception of one term). Along with the use of the Sherman--Morrisson--Woodbury (SMW) formula, yields the following formulas for $\vZ$ and $\vZ^{-1}:$
\begin{eqnarray}\label{Z_approx}
    \vZ \approx &  \tilde{\vZ}_{k} = &  \vU_{k+1} \vB_{k} (\vU_{k+1} \vB_{k})^{\top} + \vR  \\\label{invZ_approx}
    \vZ^{-1} \approx  & \tilde{\vZ}^{-1}_{k} = & \vR^{-1} - \vR^{-1} \vU_{k+1} \vB_{k}(\Id_{k} + \vB^{\top}_{k} \vB_{k})^{-1} (\vU_{k+1} \vB_{k})^{\top} \vR^{-1}. 
\end{eqnarray}
Set $\vTheta_{k} = \Id_{k+1} + \vB_{k} \vB^{\top}_{k},$ which we note is positive definite. Therefore, using Sylvester's determinant identity, we approximate $\logdet( \vZ )$ as
\begin{align*}
    \logdet( \vZ ) \approx \logdet( \tilde{\vZ}_{k} ) = \logdet (\vR) + \logdet (\vTheta_{k}).
\end{align*}
Additionally, using \eqref{genGK_relations}, we can approximate
\begin{align*}
    \| \vA \vmu - \vd \|^{2}_{\vZ^{-1}} = \| \vU_{k+1} \beta_{1} \ve_{1} \|^{2}_{\vZ^{-1}} \approx \| \vU_{k+1} \beta_{1} \ve_{1} \|^{2}_{\tilde{\vZ}^{-1}_{k}} = \| \beta_1 \ve_{1} \|^{2}_{\vTheta^{-1}_{k}}.
\end{align*}
Putting everything together, our approximation of $\bar{\calF},$ denoted by $\tilde{\calF}_{k},$ is
\begin{align}\label{eqn:approx_obj}
    \tilde{\calF}_{k} \equiv - \log \, \pi [\bar{\vtheta}] + \frac{1}{2} \logdet(\vR) + \frac{1}{2} \logdet(\vTheta_{k}) + \frac{1}{2} \| \beta_1 \ve_{1} \|^{2}_{\vTheta^{-1}_{k}}. 
\end{align}
We observed that forming $\vTheta_k$ explicitly caused numerical issues while evaluating the log determinant term. Therefore, we instead computed this term by computing the singular values of $\vB_k$, denoted $\{\sigma_j(\vB_k)\}_{j=1}^k$, and then computing $\logdet(\vTheta_k) = \sum_{j=1}^k\log(1+\sigma_j(\vB_k)^2)$.   

\paragraph{Approximation to gradient.} We can, similarly, approximate our expression for the gradient when it is evaluated at the point $\bar{\vtheta}$.  By utilizing (\ref{approx_of_A}), we have
\begin{equation}\label{eqn::approx_of_partialZ}
\begin{aligned}
    \partial_{\theta_i} \vZ \approx \widetilde{\partial_{\theta_i} \vZ}_{k} \equiv \tilde{\vA}_{k} \partial_{\theta_i} \vQ \tilde{\vA}^{\top}_{k} + \partial_{\theta_i} \vR , \qquad 1 \leq i \leq K.
\end{aligned}
\end{equation}
In view of \eqref{invZ_approx}, our approximation to ${\partial_{\theta_i} \calF} ,$ denoted by $\widetilde{\partial_i \calF}_{k},$ is
\begin{equation}\label{eqn::approx_to_partialF}
    \begin{aligned}
        \widetilde{\partial_i \calF}_{k} = &- \frac{1}{\pi[\bar{\vtheta}]} \frac{\partial \pi[\bar{\vtheta}]}{\partial \theta_{i}} + \frac{1}{2} \left\langle \tilde{\vZ}^{-1}_{k}, \widetilde{\partial_{\theta_i} \vZ}_{k} \right\rangle_{F}  \\
        &- \frac{1}{2} \left\langle ( \widetilde{\partial_{\theta_i} \vZ}_{k} ) \tilde{\vZ}^{-1}_{k}(\vA \vmu - \vd) - \vA {\partial_{\theta_i} \vmu}, \tilde{\vZ}^{-1}_{k}(\vA \vmu - \vd)  \right\rangle_{2}, \qquad 1 \leq i \leq K.
    \end{aligned}
\end{equation}
The approximation of the gradient ${\nabla \calF}$ is denoted by $\widetilde{\nabla \calF}_{k}.$ In the form written, it is not clear how the approximations lead to improved computational benefits. However, we can derive an equivalent but alternative expression that is more computationally efficient to evaluate and that is what we implement in practice. To this end, define the matrices 
\begin{equation}\label{eqn:psi}\vPsi_i^Q = \vV_k^\top \partial_{\theta_i}\vQ \vV_k, \qquad \vPsi_i^R = \vU_{k+1}^\top \vR^{-1}(\partial_{\theta_i}\vR) \vR^{-1} \vU_{k+1} \qquad 1 \leq i \leq K,\end{equation}
and the matrix $\vT_k = \vB_k^\top \vB_k$. With these definitions, we can show (see Appendix~\ref{sec:derivations})
\begin{equation}\label{eqn:term1simp} \inn{\tilde{\vZ}^{-1}_{k}}{ \widetilde{\partial_{\theta_i} \vZ}_{k}}_{F} 
= \inn{\vPsi_i^Q}{\vT_k(\vI+\vT_k)^{-1}}_F + \inn{\partial_{\theta_i}\vR}{\vR^{-1}}_F - \inn{\vB_k^\top \vPsi_i^R\vB_k}{(\vI+\vT_k)^{-1}}_F.
\end{equation} The details of computing the estimates of the objective function and the gradient are given in Algorithm~\ref{alg:objgrad}. In Section~\ref{ssec:compcosts}, we will derive computational costs and provide an error analysis.

\begin{algorithm}[!ht]
\begin{algorithmic}[1]
    \REQUIRE Parameter $\vtheta$, matrix $\vA$, and parameter $k$
    \STATE Compute matrices $\vR,\vQ$
    \STATE $[\vU_{k+1},\vB_k, \vV_{k}]=$genGK$(\vA,\vR,\vQ,\vmu,\vd,k)$;
    \STATE \COMMENT{Step 1: Estimating the objective function} 
    \STATE Compute $\vTheta_k = \vI_{k+1} + \vB_k\vB_k^\top$ and $\vT_k = \vB_k^\top \vB_k$
    \STATE Compute $\widetilde{\calF}_k $ using~\eqref{eqn:approx_obj}.
    \STATE \COMMENT{Step 2: Estimating the gradient}
    \STATE Compute $\vr_k = \widetilde{\vZ}_k^{-1}(\vA\vmu - \vd)$ using~\eqref{invZ_approx}        
    \FOR {i = 1,\dots, K}
    \STATE Compute $\vPsi_i^Q $ and $\vPsi_i^R  $ as in~\eqref{eqn:psi}
    \STATE Compute $\inn{\tilde{\vZ}^{-1}_{k}}{ \widetilde{\partial_{\theta_i} \vZ}_{k}}_{F}$ using~\eqref{eqn:term1simp}
    
    \STATE Compute $\widetilde{\partial_{\theta_i} \vZ}_{k}\vr_k = \vU_{k+1}\vB_k \vPsi_i^Q (\vU_{k+1}\vB_k)^\top \vr_k + {\partial_{\vtheta_i} \vR} \vr_k $ by exploiting the low-rank structure 
    \RETURN Approximation to objective function $\widetilde{\calF}_k$ and gradient $(\widetilde{\partial_{\theta_i}\calF})_{i=1}^K$.
    \ENDFOR
\end{algorithmic}
\caption{genGK: Approximation to objective function and gradient}
\label{alg:objgrad}
\end{algorithm}

\subsection{Computational Cost}\label{ssec:compcosts}

 To derive an estimate of the computational cost, we have to make certain assumptions. First, assume that the computational cost associated with computing the forward operator or its adjoint, $\vA \vx$ or $\vA^{\top} \vy$ respectively (where $(\vx, \vy) \in \R^{n} \times \R^{m}$), can be modeled by $\calT_{\vA}.$ Additionally, assume that the cost associated with computing $\vQ \vx$ or $\partial_{\theta_i} \vQ  \vx$ for any $1 \leq i \leq K $ is $\calT_{\vQ}.$  Finally, we take $\vR$ to be a diagonal matrix which assumes that the noise is uncorrelated.

 With these assumptions in place, the total cost associated with executing Algorithm~\ref{algorithm:genGK} is 
\begin{align*}
    2 (k + 1) \calT_{\vA} + (2k + 1) \calT_{\vQ} + O(k(m + n)) \text{ flops},
\end{align*}
where ``flops'' refers to floating point operations. This represents the dominant cost of our approach.

Once the genGK matrices have been computed, it is straightforward to compute the objective function and the gradient. For the objective function, we need to compute the SVD of $\vB_k$, so the additional cost of computing $\tilde{\calF}_{k}$ is $O(k^3)$ flops. The cost of computing the gradient is additionally $K(k \calT_{\vQ} + O(k^2(m + n) + k^3))$ flops.

We summarize these results in Table~\ref{tab:compcosts}.

\begin{table}[h!] \centering 
\begin{tabular}{ c|c|c|c } 
 & genGK & $\tilde{\calF}_{k}$ &  $\widetilde{\partial \calF}_{k}$ \\ 
 \hline
 Involving $\vA$ & $2 (k + 1) \calT_{\vA}$ & $0$ & 0 \\ 
 \hline
 Involving $\vQ$ (and derivatives) & $(2k + 1) \calT_{\vQ}$ & $0$ & $kK \calT_{\vQ}$
\\ \hline
Additional & $O(k(m + n))$ & $O(k)$ &  $O(k^2(m + n) + k^3))$
\end{tabular}
\caption{Computational costs associated with Algorithm \ref{algorithm:genGK}, $\tilde{\calF}_{k},$ and $\widetilde{\partial \calF}_{k}.$ Here $\calT_{\vA}$ and $\calT_{\vQ}$ model the costs associated with computing matrix-vector products with $\vA$ and $\vQ$ (or $\partial_{\theta_i} \vQ $) respectively.}
\label{tab:compcosts}
\end{table}

\subsection{Error analysis of genGK approximations}
The following two propositions quantify the errors between $\calF$ and an approximation of $\calF$ when the approximation is generated by a low-rank representation of the matrix $\vZ.$ These results provide insight into the accuracy of the genGK process and help develop an {\em a posteriori} error estimator to monitor the error in $\calF_k$.
Let  $\sigma_{1}(\widehat{\vA}) \geq \sigma_{2}(\widehat{\vA}) \geq \cdots \geq \sigma_{r}(\widehat{\vA}) > \sigma_{r + 1}(\widehat{\vA}) = \cdots = \sigma_{\min \, \{ m, n \}}(\widehat{\vA}) = 0$ be the singular values of $\widehat{\vA}$ and  $r = \rank(\widehat{\vA})$. Let $\widehat{\vA}_{k} = \vU \vSigma_{k} \vV^{\top}$ be the $k$-rank approximation of the matrix $\widehat{\vA} = \vR^{-1/2} \vA \vQ^{1/2}$  for $1 \leq k < r.$

\begin{proposition}
\label{prop::svd_bound}
Define the approximation to $\bar\calF$ using the truncated SVD of $\widehat{\vA}$ as
\begin{equation}
\begin{aligned}
    \widehat{\calF}_{k} = - \log \, \pi[\bar{\vtheta}] + \frac{1}{2} \logdet \, (\vZ_{k}) + \frac{1}{2} \| \vA \vmu - \vd \|_{\vZ^{-1}_{k}}^2,
\end{aligned}
\end{equation}
with $\vZ_{k} = \vR^{1/2} \left[ \widehat{\vA}_{k} \widehat{\vA}_{k}^{\top} + \Id_{m} \right] \vR^{1/2}.$ Then
\begin{equation}
\begin{aligned}
    |\bar{\calF} - \widehat{\calF}_{k}| \leq \frac{1}{2} \sum\limits_{i = k + 1}^{r} \log \, (1 + \sigma^{2}_i(\widehat{\vA})) + \frac{1}{2} \beta^{2}_{1} \left( \frac{\sigma^{2}_{k+1}}{1 + \sigma^{2}_{k+1}} \right),
\end{aligned}
\end{equation}
where $\beta_{1} = \| \vA \vmu - \vd \|_{\vR^{-1}}$.
\end{proposition}
\begin{proof}
See Appendix~\ref{ssec:svd}.
\end{proof}

This proposition says that if the singular values of $\widehat{\vA}$ (alternatively, generalized singular values of $\vA$) decay rapidly, then the low-rank approximation $\widehat{\vA}_k$ obtained using the truncated SVD results in an accurate approximation of the objective function. Indeed, if $\rank(\widehat{\vA}) = k$, then $\widehat{\calF}_k = \calF$. However, repeated computation of the truncated SVD of $\widehat{\vA}$ at each optimization step is computationally expensive. As numerical experiments will demonstrate, the low-rank approximation obtained using genGK is accurate and is also computationally efficient. The next result derives an error estimate for the absolute error in the approximation of the objective function, where the approximation is computed using genGK. 

\begin{proposition}
\label{prop:genGK_bound}
Consider $\bar{\calF}$ and $\tilde{\calF}_{k}$ as defined in (\ref{calF_bar}) and (\ref{eqn:approx_obj}) respectively. Define the orthogonal projector, $\vPi_{\vV_{k}}\equiv \vQ^{1/2} \vV_{k} (\vQ^{1/2} \vV_{k})^{\top}$ and let $\vH_{\vQ} \equiv \vQ^{1/2} \vA^{\top} \vR^{-1} \vA \vQ^{1/2}$ and ${\vH}^{(k)}_{\vQ} =  \vPi_{\vV_{k}}\vH_{\vQ}\vPi_{\vV_{k}} .$ Then the absolute error in the objective function satisfies
\begin{align} 
    |\bar{\calF} - \tilde{\calF}_{k}| \leq \frac{1}{2} \left[ \xi_{k} + \beta^{2}_{1} \left( \frac{\xi_{k}}{1 + \xi_{k}} \right) \right],
\end{align}
where $\xi_k = \trace \left( \vH_{\vQ} \right) - \trace \left( {\vH}^{(k)}_{\vQ} \right),$ and  $\beta_1 = \| \vd - \vA \vmu \|_{\vR^{-1}}$.
\begin{proof}
See Appendix~\ref{ssec:proofgengk}.
\end{proof}
\end{proposition} 

By \cite{AKS_Chung_Petroske_EffKrylov}, we have the following recurrence for $\xi_k= \trace \left( \vH_{\vQ} - {\vH}^{(k)}_{\vQ} \right).$ For any $k < \min\{ m, n \} - 1,$ $\xi_k$ satisfies 
    \begin{equation*}
        \xi_{k+1} = \xi_{k} - ( \alpha^{2}_{k+1} + \beta^{2}_{k+2} ).
    \end{equation*}
This shows that the term $\xi_k$ is monotonically nonincreasing.  Therefore, Proposition~\ref{prop:genGK_bound} says that the absolute error can be bounded by a monotonically increasing function. Furthermore, this proposition gives us a way to monitor the accuracy of the objective function and determine a suitable stopping criterion for terminating the iterations. This is discussed next.

\subsection{Monitoring convergence of genGK}\label{ssec:monit}
As mentioned earlier, an important issue that needs to be addressed is a stopping criterion to determine the value of $k$ with which to approximate the objective function and the gradient.

Proposition~\ref{prop:genGK_bound} can be converted to a stopping criterion to monitor convergence, but the issue is that the initial iterate $\xi_0 = \trace(\vH_\vQ)$ cannot be computed efficiently in a matrix-free fashion. To address this issue we use a combination of the Monte Carlo trace estimator and the genGK recurrence to monitor $\xi_k$; then in combination with the upper bound in Proposition~\ref{prop:genGK_bound} we can monitor the convergence. 

\paragraph{Monte Carlo estimator.} First, we describe the Monte Carlo estimators for the trace of a matrix. Let $\vK \in \R^{n\times n}$ be a square matrix and suppose we are interested in estimating its trace using matrix-free techniques. Let $\vomega \in \R^n$ be a random vector with mean zero and identity covariance matrix. Then note that 
\[ \expect[\vomega^\top\vK\vomega] =  \expect[\trace(\vomega^\top\vK\vomega)] = \trace(\vK\expect[\vomega\vomega^\top]) = \trace(\vK),  \]
where we have used the cyclic property of trace. Therefore, based on this identity, we can use a Monte Carlo estimator for $\trace(\vK)$ as 
\[ \trace(\vK) \approx \frac{1}{n_{\rm mc}}\sum_{j=1}^{n_{\rm mc}} \vomega_j^\top \vK\vomega_j = \frac{1}{n_{\rm mc}}\trace(\vOmega^\top \vK\vOmega) \equiv \widehat{\trace}(\vK), \]
where $\vomega_j$ for $1 \leq j \leq n_{\rm mc}$ are independent random vectors from the distribution of $\vomega$ and $\vOmega = \begin{bmatrix} \vomega_1 & \dots & \vomega_{n_{\rm mc}}\end{bmatrix}$.

This next result analyzes the error in the trace estimator. 
\begin{proposition}\label{prop:trace} Let $\vomega_k \in \mathbb{R}^n$ for $1 \leq k \leq N$ be random vectors with independent sub-Gaussian entries that have zero mean and $\max_j \| (\vomega_k)_j\|_{\psi_2} \leq K_\psi$ (see~\cite[Section 2.5.2]{vershynin2018high} for a definition of the norm $\|\cdot \|_{\psi_2})$. The error in the trace estimator satisfies the probabilistic bound
\begin{equation}\label{eqn:probtrace} \mathbb{P}\left\{| \widehat{\trace} (\vH\vQ) - \trace(\vH_{\vQ})| \geq t \right\} \leq 2\exp\left( - C_{\rm HW} N \min\left\{ \frac{t^2}{K_\psi^4\|\vH\vQ\|_F^2}, \frac{t}{K_\psi^2\|\vH\vQ\|_2}  \right\} \right) , \end{equation}
where $C_{\rm HW}$ is an absolute constant. Furthermore, let $\epsilon > 0$ be the desired relative error and  $ 0 < \delta < 1$  the desired failure probability. Then with probability at least $1-\delta$, the number of samples $N_{\rm samp}$ must satisfy 
\[ N_{\rm samp} \geq \frac{K_\psi^2 \log(2/\delta)}{C_{\rm HW} \epsilon^2 }
\left( \frac{K_\psi^2\|\vH\vQ\|_F^2}{\trace(\vH\vQ)^2}  +  
\frac{\epsilon \|\vH\vQ\|_2}{ \trace(\vH\vQ)}\right) \]
 to ensure $| \widehat{\trace} (\vH\vQ) - \trace(\vH\vQ)| \leq \epsilon |{\trace} (\vH\vQ)|$.
\end{proposition}
The novelty in this result is that it does not require $\vK$ to be symmetric or positive semidefinite as is the standard assumption (see, e.g.,~\cite{cortinovis2021randomized} and references within). 
\begin{proof}
    See Appendix~\ref{ssec:trace}.
\end{proof}
In practice, we take $\vOmega$ to be a standard Gaussian random matrix, i.e., with independent and identically distributed entries drawn from $\mathcal{N}(0,1)$. {Another option is} to choose the entries of the random matrix to be independent Rademacher random variables (entries $\pm1$ with equal probability). The analysis in Proposition~\ref{prop:trace} applies to both distributions since both distributions are sub-Gaussian with sub-Gaussian norm $K_\psi = \mathcal{O}(1)$.

\begin{algorithm}[!ht]
    \begin{algorithmic}[1]
        \REQUIRE Matrices $\vH$, $\vQ$, $\vV_k$, $\vT_k$. Integers $k_{\max}$, $n_{\rm mc}$. 
        \STATE Draw $\vOmega \in \R^{n\times n_{\rm mc}}$ with i.i.d. entries from $\mathcal{N}(0,1)$
        \STATE Compute  $\vY = \vH\vQ\vOmega$
        \FOR {$k = 1,\dots,k_{\max}$}
        \STATE Compute $\vY_{k} = \vY - \vV_k\vT_k\vV_k^\top\vQ\vOmega$
        \STATE Estimate $\widehat\xi_k  = \trace(\vOmega^\top\vY_k)/n_{\rm mc}$
        \ENDFOR
        \RETURN Estimate $\{\widehat\xi_k\}_{k=1}^{k_{\max}}$.
    \end{algorithmic}
    \caption{Monitoring accuracy of the genGK approximations}
    \label{alg:monitor}
\end{algorithm}

\paragraph{Monitoring error.} Next, we explain how to estimate $\xi_k$ using the Monte Carlo trace estimator. First consider $\xi_k$, which can be written using the cyclic property of the trace as 
\[ \xi_k = \trace(\vH_{\vQ}) - \trace(\vH_{\vQ}^{(k)}) = \trace(\vH\vQ - \vV_k\vT_k\vV_k^\top \vQ).\]
Note that in this formulation, we do not need to work with the square roots of $\vQ$. Applying the Monte Carlo trace estimator with $\vOmega$ as a standard Gaussian random matrix, we can estimate $\xi_k$ as 
\[ \widehat\xi_k : = \frac{1}{n_{\rm mc}}\trace( \vOmega^\top (\vH\vQ - \vV_k\vT_k\vV_k^\top \vQ)\vOmega). \] 
Note that this algorithm only requires matrix-vector products with $\vH$ and $\vQ$ and the genGK relationships.  The details and efficient implementation of this estimator are given in Algorithm~\ref{alg:monitor}. In practice, we take the number of Monte Carlo samples $n_{\rm mc} = 10$. Finally, we use the estimator 
\[ \text{err}_{\rm mc} := \frac12 \left[ \widehat\xi_k  + \beta_1^2\frac{\widehat\xi_k}{1 + \widehat{\xi}_k}\right],\]
to estimate the error in the objective function. Note that while the upper bound in Proposition~\ref{prop:genGK_bound} no longer holds with the Monte Carlo estimators, it can still be used to monitor the accuracy as an error indicator.

\section{Numerical Experiments}
\label{sec:numerics}

In the following, we outline experiments corresponding to problems in one and two dimensions that arise in inverse problems. Using our approach, we demonstrate that approximations to both the objective function $\calF$ and its gradient $\nabla \calF$ using the genGK relations can lead to hyperparameter estimates that result in accurate reconstructions.
We also show that one can effectively inform the selection of the bidiagonalization parameter $k$ using the error bounds derived above. Tests of robustness are presented, and for the one-dimensional case, illustrations for the computational need for a framework such as ours are provided.

All experiments are performed using MATLAB, with optimizations done via \verb+fmincon+ with an interior point method~\cite{interior_point_1,interior_point_2,interior_point_3}.

\paragraph{Choice of priors and hyperpriors.} Although our approach is general, for concreteness we assume that the dimensionality of our hyperparameter space is three (i.e., $K = 3$). The first hyperparameter controls the variance of the noise. We assume that the noise is Gaussian, with zero mean and covariance $\vR(\vtheta) = \theta_1 \Id_{m}$.  We model the prior covariance matrix using the Mat\'{e}rn covariance family. This is a flexible family of covariance kernels that can be used to model a wide range of behaviors with a relatively few set of parameters.  The kernel is isotropic and is  defined using the covariance function 
\[ \calM_{\nu, \sigma^2, \ell}(r) \equiv \frac{\sigma^{2}}{2^{\nu - 1} \Gamma(\nu)} \left(  \sqrt{2 \nu}\frac{r}{\ell} \right)^{\nu} \calK_{\nu} \left(  \sqrt{2 \nu} \frac{r}{\ell} \right)\] 
where $\Gamma$ is the gamma function,  $\calK_{\nu}$ is the modified Bessel function of the second kind, and the positive parameters $\nu, \sigma^2, \ell $ represent the smoothness of the process, the prior variance, and the correlation length respectively. Given the covariance function $\calM_{\nu,\sigma^2,\ell}(r)$, the covariance kernel is given by $\kappa(\vx, \vy) = \calM_{\nu,  \sigma^2, \ell}(\|\vx - \vy\|_2)$. Given a set of points $\{\vx_j\}_{j=1}^n$ the covariance matrix $\vQ(\vtheta)$ can be constructed entrywise as 
\[ [\vQ(\vtheta)]_{i,j}  = \calM_{\nu,\theta_2^2, \theta_3}( \| \vx_i - \vx_j\|_2) \qquad 1 \leq i, j \leq n.\] 
Therefore, $\theta_2$ and $\theta_3$ represent the prior standard deviation and the correlation length. We do not estimate the smoothness $\nu$ as a part of the estimation process but assume that it is fixed. 
We also assume the prior mean $\vmu(\vtheta) = \vzero$. In all the numerical experiments, we use the FFT-based technique to compute matvecs in $\mathcal{O}(n\log n)$~\cite{NowakTenkleveCirpka}.

For the hyperpriors, following~\cite{bardsley2018computational}, we report experiments with two different choices: (P1) is the improper prior chosen as $\pi(\vtheta) \propto 1$; and (P2) is a Gamma prior with  
\begin{equation}
\pi(\vtheta) \propto \exp\left( - \sum_{j=1}^K\gamma \theta_i \right) \qquad \theta_i > 0, \quad 1 \leq i \leq K. 
\end{equation}
The parameter $\gamma$ is set to be $10^{-4}$ and chosen such that the probability density function is relatively flat over the parameter space. We experimented with other hyperpriors and obtained similar results.

\paragraph{Experimental setup.} To generate the noise we adopt the following procedure. A linear forward operator $\vA$ along with some exact true signal $\vs$ is constructed so that the uncorrupted data, $\vd_{*},$ is $\vd_{*} = \vA \vs.$ Then we generate a realization, $\vepsilon,$ of a Gaussian random vector $X_{\vepsilon} \sim \calN(\vzero, \Id_{m}).$ Next, we set $$\veta = \vepsilon \frac{\lambda_{\textrm{noise}} \| \vd_{*} \|_{2}}{\| \vepsilon \|_{2}},$$ where $\lambda_{\textrm{noise}} \in (0,\infty)$ represents a noise level parameter; by construction, it is such that $\| \veta \|_{2} = \lambda_{\textrm{noise}} \| \vd_{*} \|_{2}.$ Finally, we generate a vector of noisy data observations via $\vd = \vd_{*} + \veta.$ 

We compare the accuracy in terms of relative reconstruction error norms, defined as 
\[ \text{RE} = \frac{\|\vs - \widehat{\vs}\|_2}{\|\vs\|_2},\]
where $\vs$ is the true solution and $\widehat{\vs}$ is the approximation. The timing computations were performed on North Carolina State University's High Performance Computing cluster `Hazel' using MATLAB R2023a. Specifically, our computing resource utilized an Intel Xeon Gold 6226 microprocessor with 188 GB of RAM, 2 sockets, and 32 cores per socket.

\subsection{Application 1: Inverse heat transfer}

Our first experiment corresponds to a one-dimensional problem in heat conduction, the details of which are outlined in~\cite{Engl_RegularizationInverseProblems,ReguToolbox,Carasso_1dHeat,Elden_1,Brunner_Collocation}.  Let $L, \kappa >0,$ and define
\begin{align}
    K : [0,L]\times [0,L] \rightarrow \R : (t,s) \mapsto \frac{1}{\sqrt{ 4 \pi \kappa^2 }} \lrp{s - t}^{-3/2} \exp \left( - \frac{1}{4} \lrp{s - t}^{-1} \right),
\end{align}
and for any $t \in [0,L],$ define the integral operator
\begin{align}
    (T \varphi)(t) = \int_{0}^{t} K(t,s) \varphi(s) \, ds. \label{volterra_first_kind_eq}
\end{align}
with the parameter $\kappa = 1.$ The parameter $\kappa$ controls the degree of ill-posedness of the problem; here, $\kappa = 1$ yields an ill-posed problem, whereas $\kappa = 5$ gives a well-posed problem. The inverse problem consists of determining a function $f$ such that 
    $T f = g$, 
{where} $g$ is known. 

Using the Regularization Tools package~\cite{ReguToolbox}, a discretization of $T$, in the domain $\Omega = [0,1]$, is generated. We denote the discretized representation of $T$ as $\vA \in \R^{n\times n}$ and the approximation $\vs$ for the function $f$.  For the prior covariance, we choose the Mat\'ern covariance with $\nu = 3/2$ and for the hyperpriors, we used the improper prior (P1). The measurements were corrupted with $2\%$ additive Gaussian noise.

\paragraph{Experiment 1: Accuracy of objective function.}
In this experiment, we investigate the accuracy of the genGK estimates in computing the objective function at the optimal value of $\vtheta$ (this will be discussed in Experiment 2). 
In the left plot of Figure~\ref{fig::error_bounds_1D}, we plot the relative error of $\widetilde{\calF}_{k}$ for various $k$ (blue, solid line) and compare it against the estimated bound derived using Monte Carlo techniques with $n_{\rm mc} = 10$ (red, dashed line). 

\begin{figure}[!ht]
    \centering
    \includegraphics[scale=0.21]{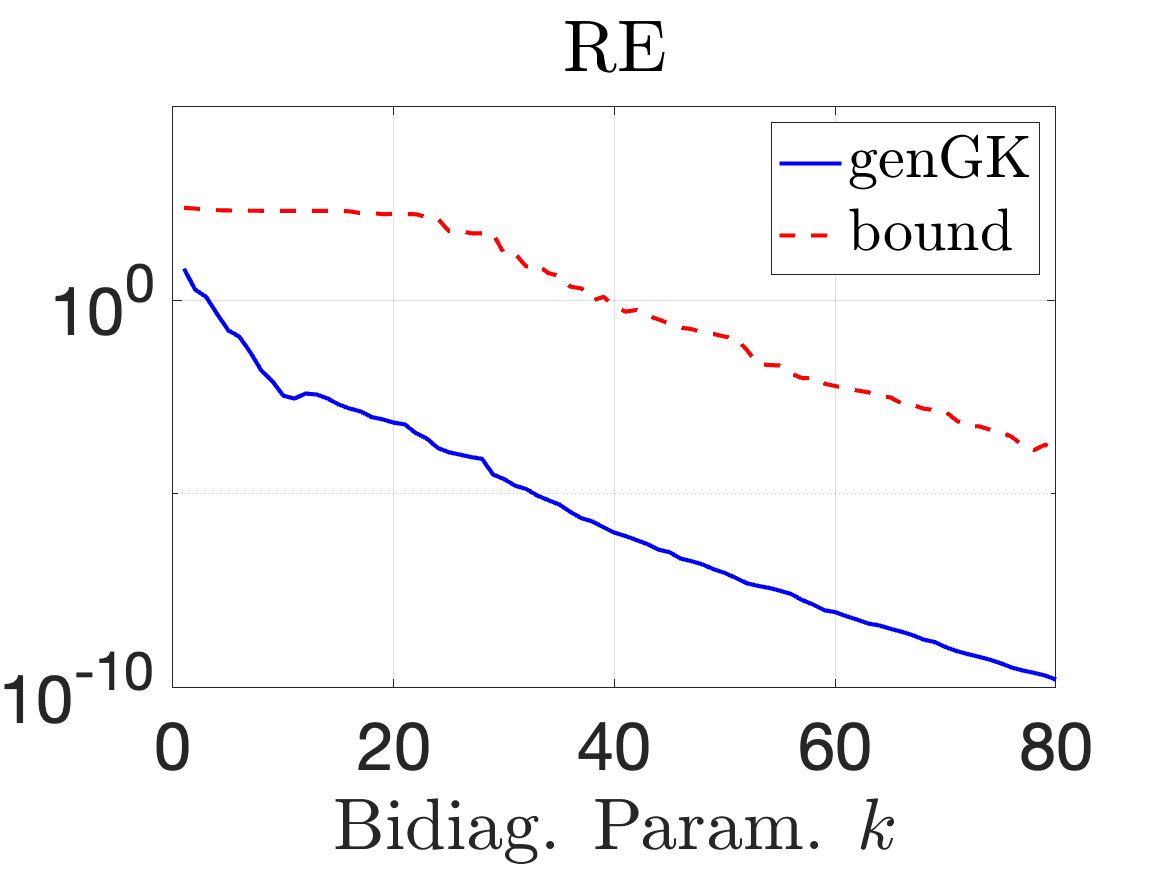}
    \includegraphics[scale=0.21]{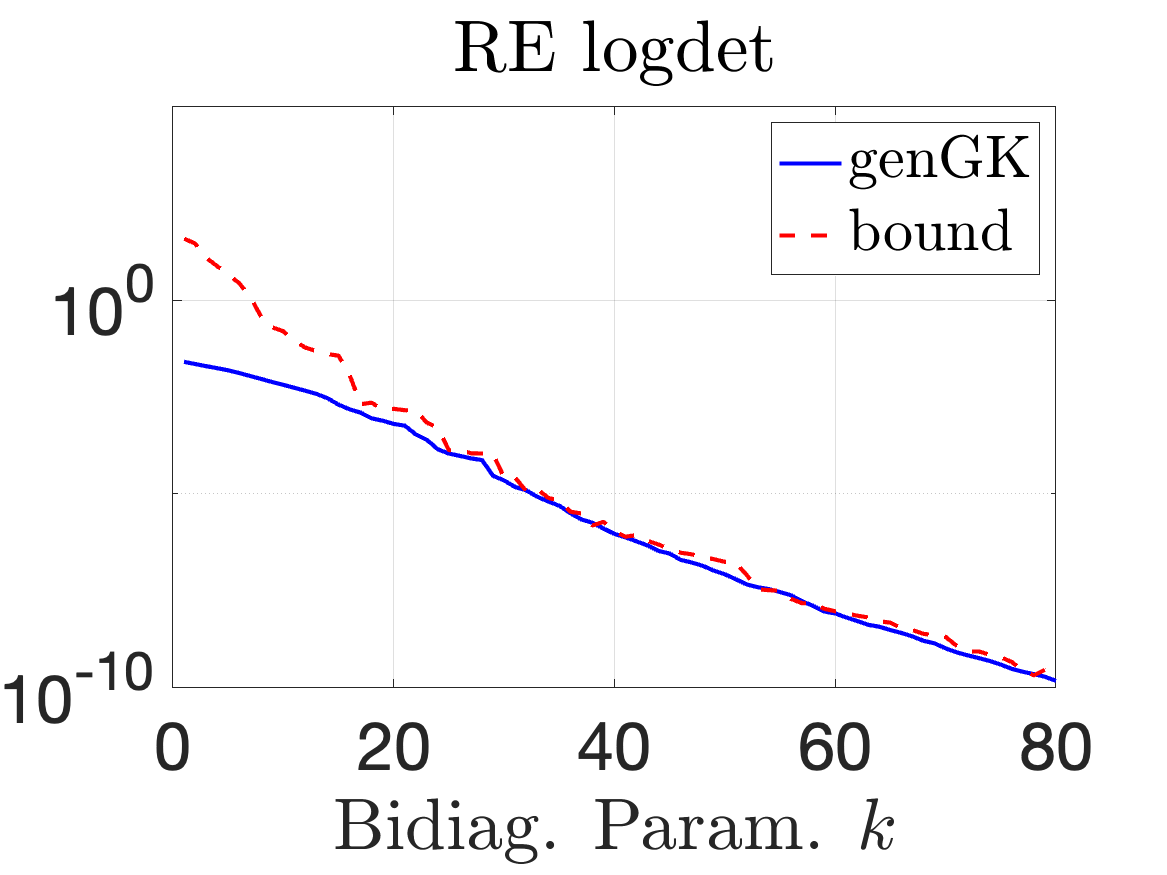}
    \includegraphics[scale=0.21]{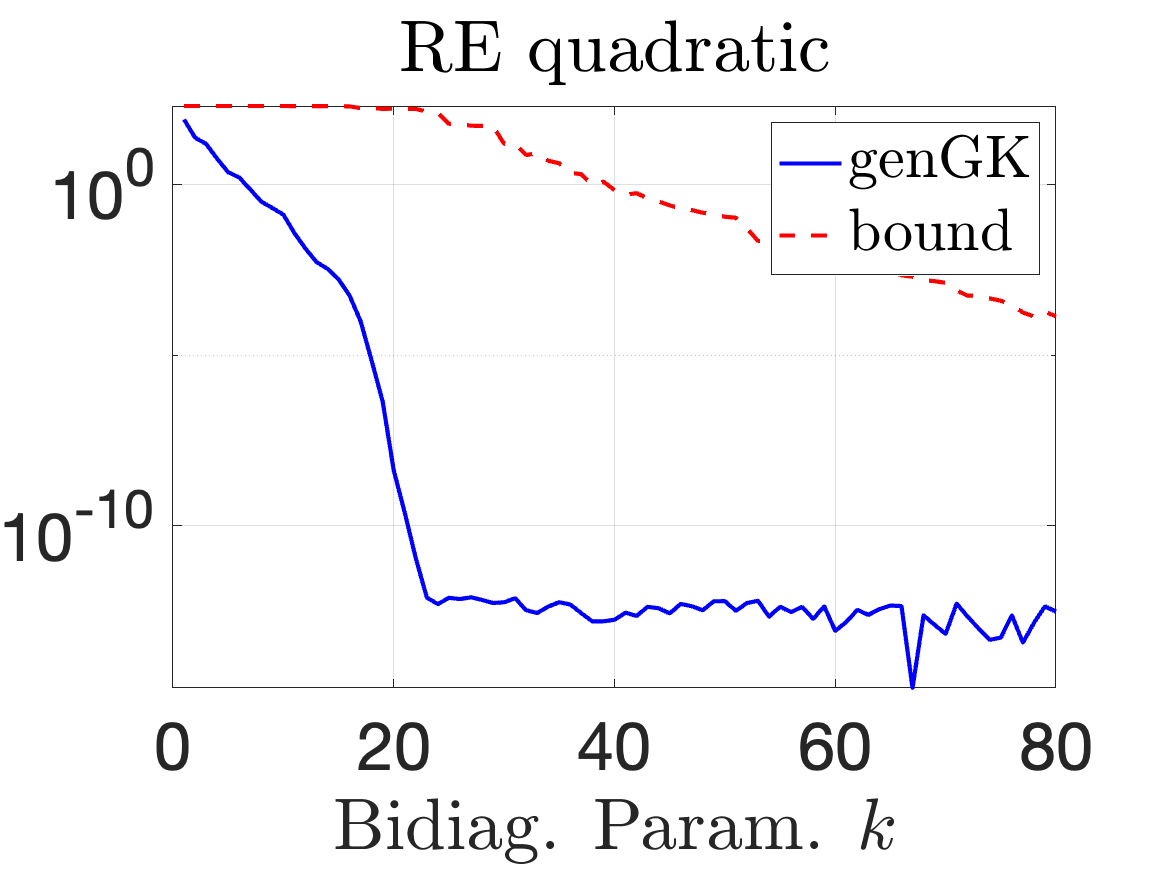}
    \caption{1D Heat: Relative errors and corresponding bounds. The left panel compares the overall accuracy of the objective function; the middle and right panels compare the errors of two components of the objective function: the log determinant and the quadratic terms respectively.}
    \label{fig::error_bounds_1D}
\end{figure}

We see from Figure~\ref{fig::error_bounds_1D} (left panel) that the computable bound reasonably tracks the observed errors over a large range of $k$ values. 
To investigate further, we separated the main contributions to the objective function: the log-determinant term $\frac12 \logdet(\vZ(\vtheta))$ and the quadratic term $\frac12\|\vA\vmu(\vtheta) -\vd\|_{\vZ(\vtheta)^{-1}}^2$. The relative errors in these terms are defined as 
\[ {\rm RE}_{\rm logdet} =  \frac{|  \logdet(\vZ) -  \logdet(\tilde{\vZ}_k)|}{|\logdet(\vZ)|} \qquad {\rm RE}_{\rm quad} =   \frac{|  \|\vA\vmu -\vd\|_{\vZ^{-1}}^2  - |\vA\vmu -\vd\|_{\tilde{\vZ}_k^{-1}}^2 |}{\|\vA\vmu -\vd\|_{\vZ^{-1}}^2}. \]
In the above expressions, we have suppressed the dependence on $\vtheta$. Correspondingly the appropriate computable bounds are ${\rm CB}_{\rm logdet} = \frac{ \hat\xi_k}{2\calF}$ and ${\rm CB}_{\rm quad} = \frac{\beta_1^2\hat\xi_k}{2\calF (1+\hat\xi_k)}$; see Section~\ref{ssec:monit}.
From the middle and right plots of Figure~\ref{fig::error_bounds_1D},
we observe that the bound for the log determinant term is better than for the quadratic term. This is explored further in Experiment 2. 

Moreover, we observe that the errors in the objective function approximations exhibit sharp decay with increasing values of $k$. To explain this behavior, we plot the generalized singular values of the operator. We can observe that the rank-$k$ SVD approximation produces comparable results to the genGK approximation. Note that the errors using these approaches are not guaranteed to be monotonically decreasing. Nevertheless, the sharp decay in the error can be understood by considering the decay in the generalized singular values in Figure~\ref{fig::GenSingVals}.   

\begin{figure}[!ht]
    \centering
    \includegraphics[scale=0.3]{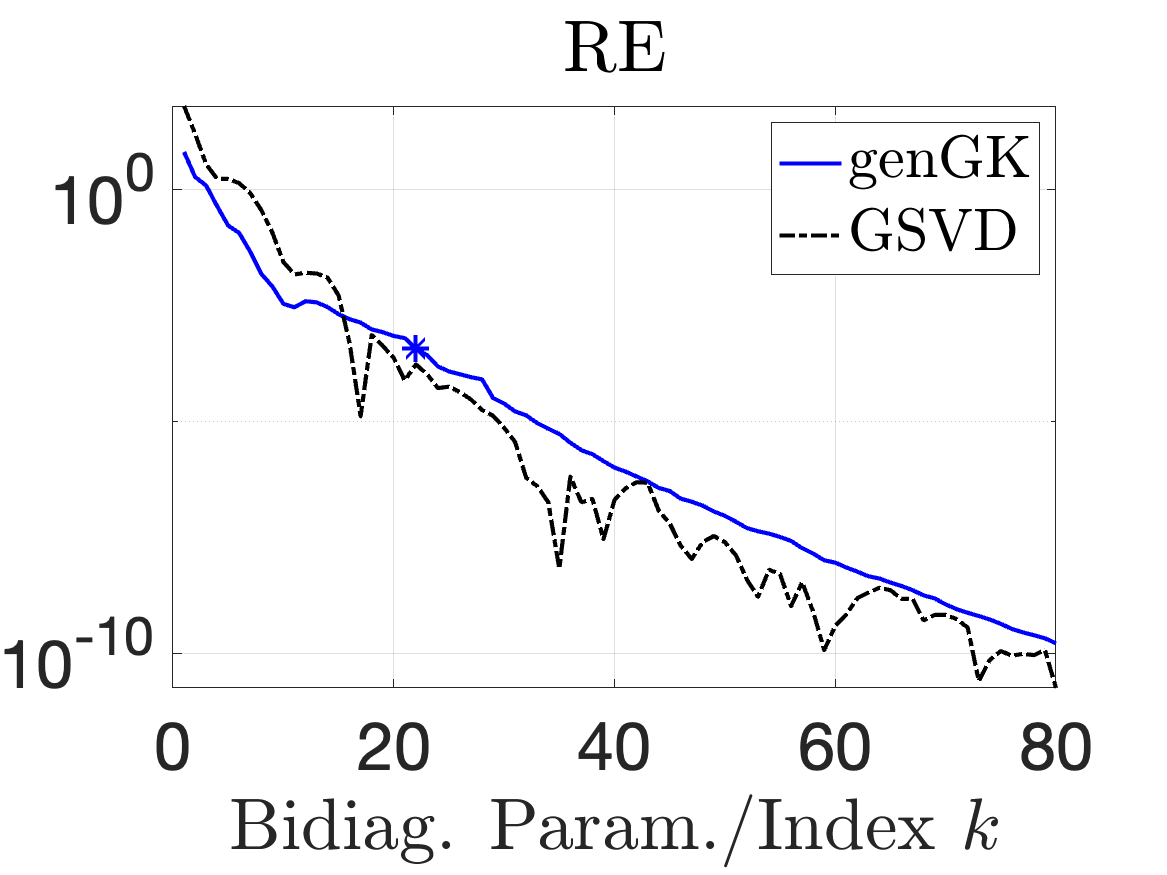}
    \qquad
    \includegraphics[scale=0.3]{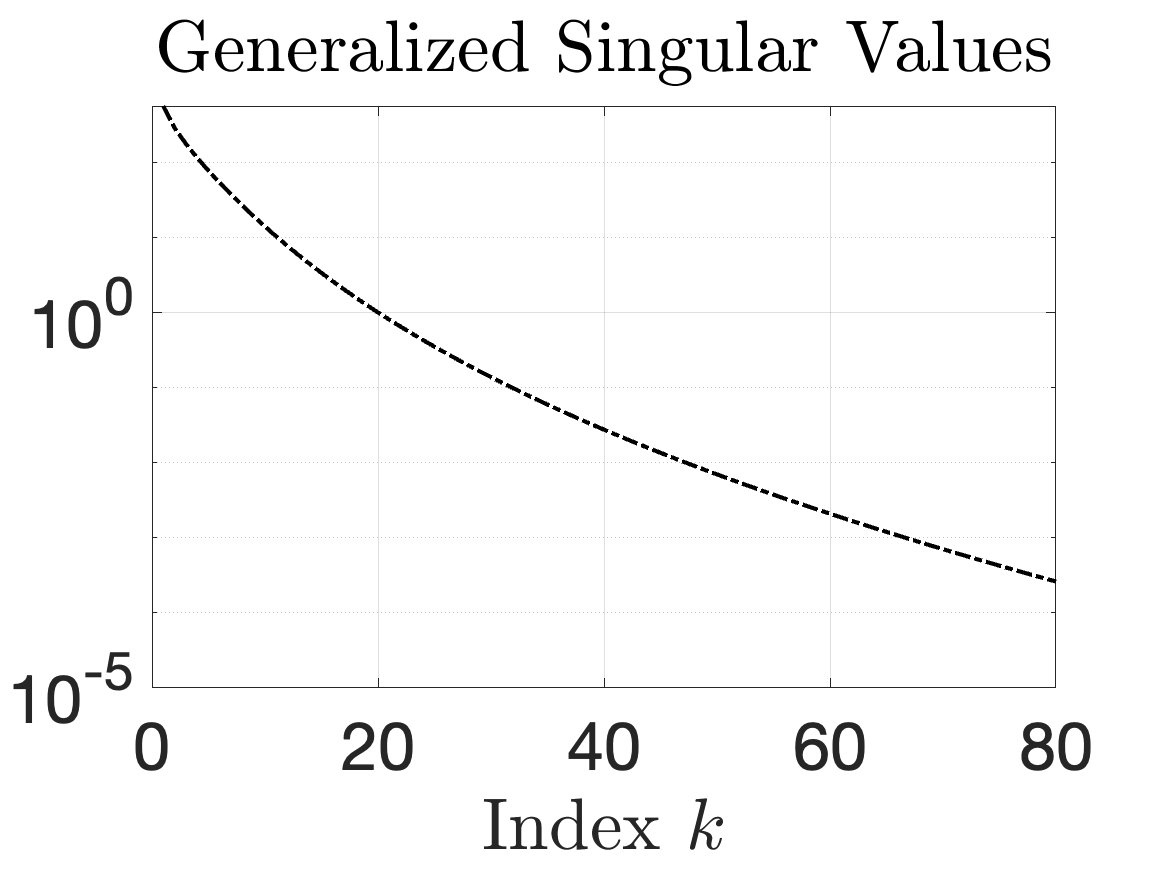}
    \caption{1D Heat: (left) Relative error in the objective function using genGK and GSVD, and (right) plot of singular values of $\widehat\vA$.}
    \label{fig::GenSingVals}
\end{figure}

The blue star in the left plot denotes the $k$ value that is used in our subsequent numerical experiments ($k = 22$). The relative error of the objective function at this point is approximately $10^{-4}$, with the relative error in the quadratic term being approximately $10^{-11}.$  An empirical justification for this choice is illustrated in the following sections.

\paragraph{Experiment 2: Recovery and accuracy along optimization trajectory.}
In this experiment, we discuss the performance of the optimization solver, the recovery, and the accuracy of the objective function along the optimization trajectory. We use the same setup as before and fix the number of genGK iterations to $k=22$.

\begin{figure}[!ht]
    \begin{center}
    \subfloat[\centering Noisy data with $2\%$ noise]{{\includegraphics[width=6cm]{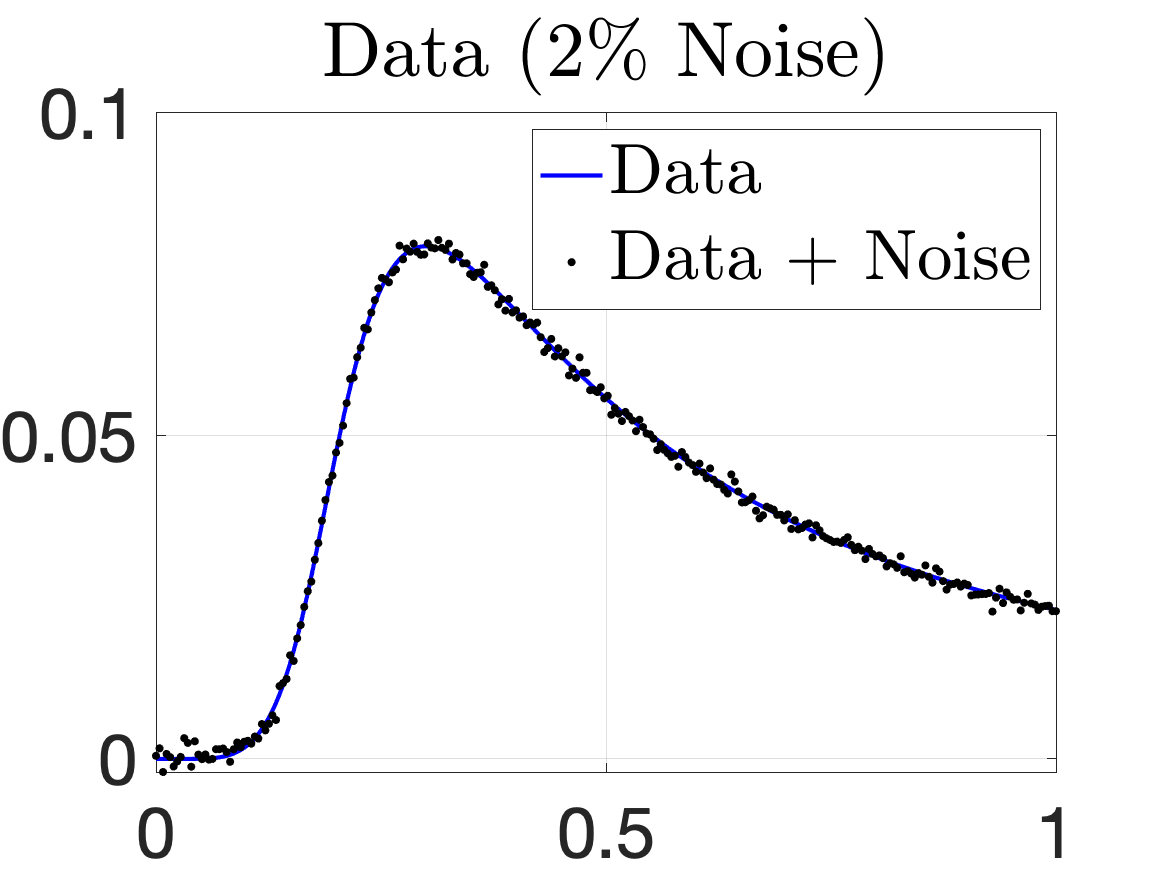} }}%
    \qquad
    \subfloat[\centering Reconstructed solution]{{\includegraphics[width=6cm]{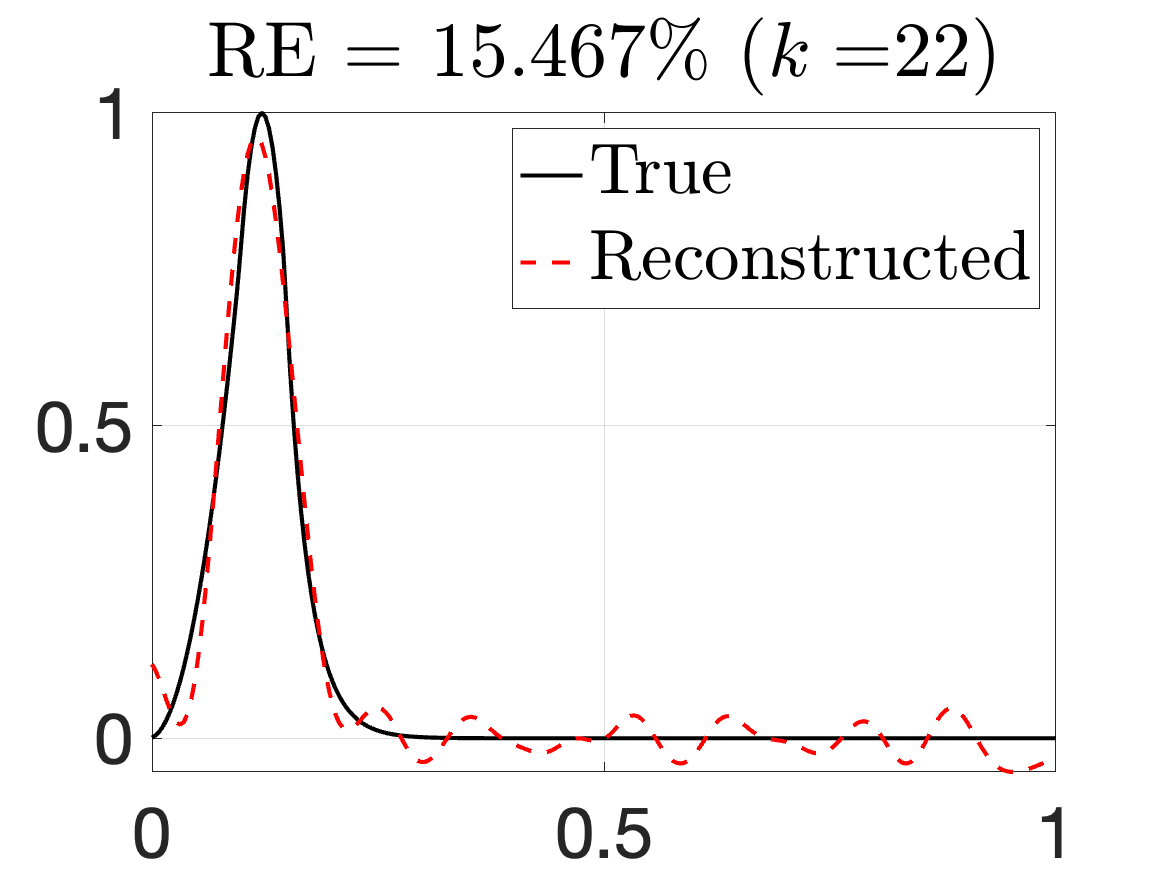} }}%
    \end{center}
    \caption{1D Heat: On the left, we provide the observations (with and without noise).  On the right, we provide the reconstruction, along with the true solution.}
    \label{fig:1d_recon}
\end{figure}

In Figure~\ref{fig:1d_recon}{(a)}, we plot the data without noise and with added noise of $2\%$. In Figure~\ref{fig:1d_recon}{(b)}, we plot a reconstruction alongside the true solution. The optimizer took $29$ iterations to converge with $69$ function and gradient evaluations (i.e., \verb+funcCount = 69+) and the optimal solution was found to be $$\vtheta^* = { (8.73\times 10^{-7}, 0.2562, 0.0566)^{\top}}.$$ 
The relative reconstruction error is found to be about $15.46\%$ at the optimal value of $\vtheta^*$.

\begin{figure}[!ht]
    \centering
    \includegraphics[scale=0.21]{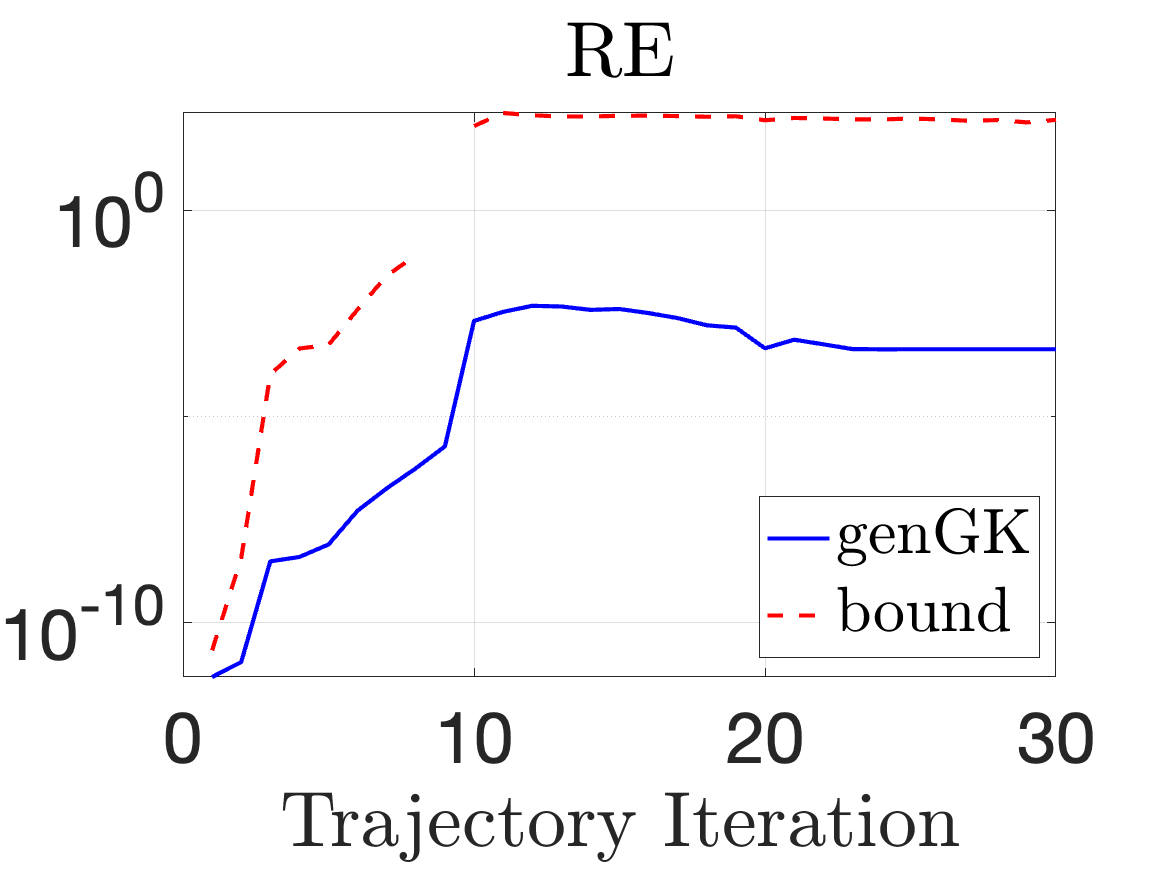}
    \includegraphics[scale=0.21]{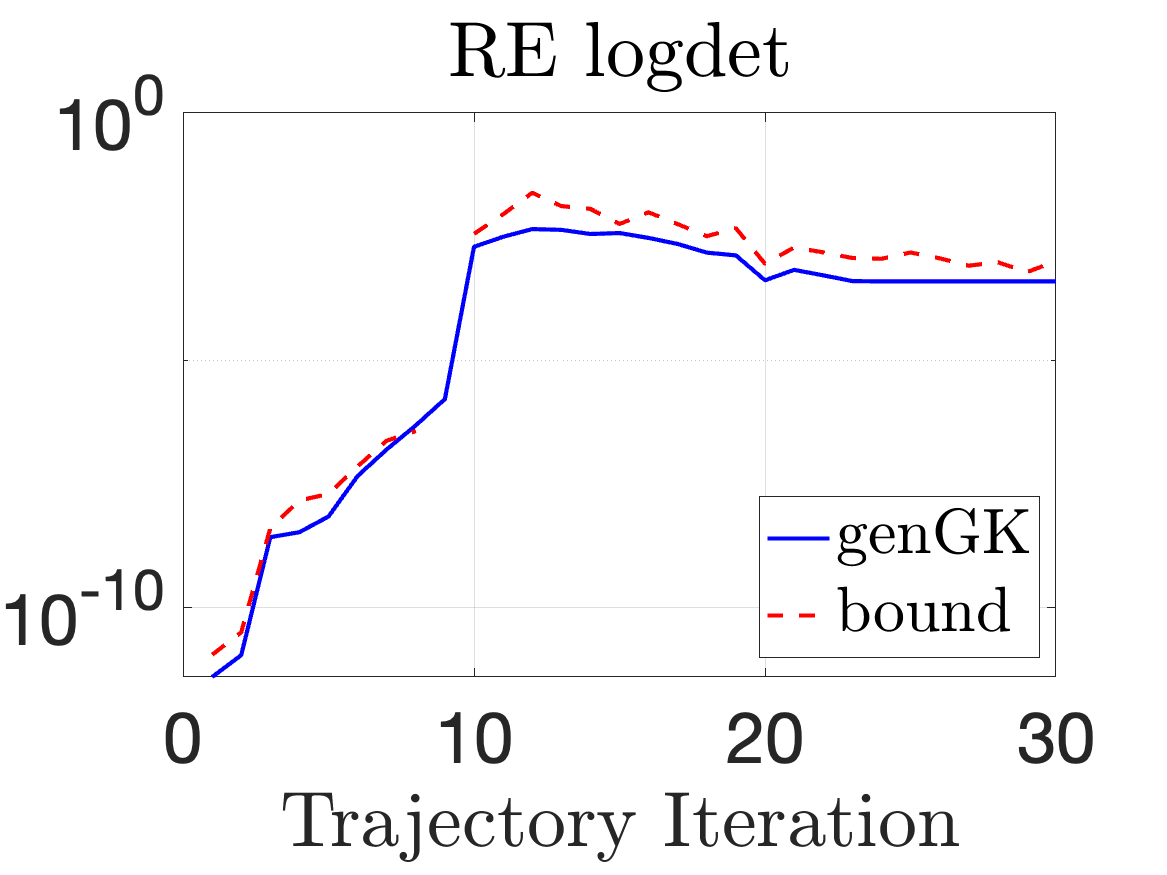}
    \includegraphics[scale=0.21]{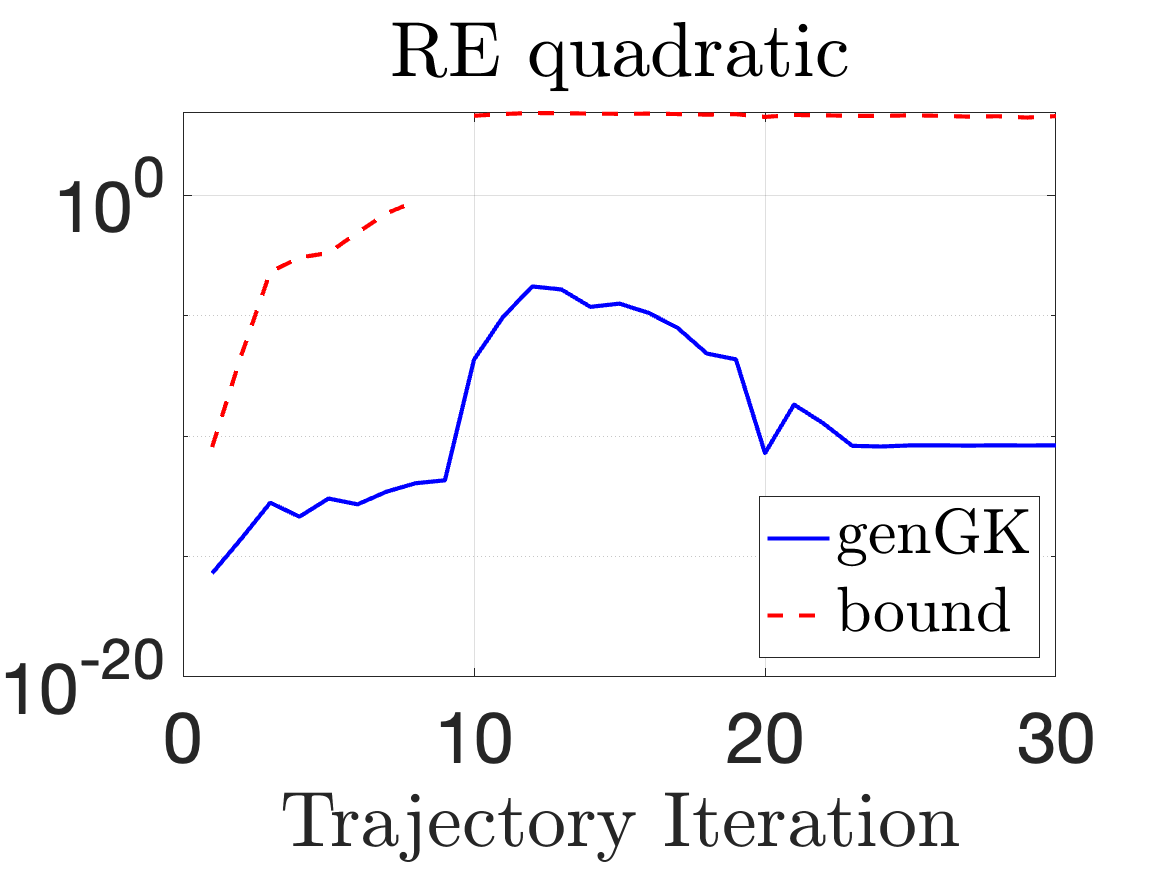}
    \caption{1D Heat: Relative errors and corresponding bounds along the optimization trajectory for $k=22$ and $n_{mc}=10$. The left panel compares the overall accuracy in the objective function, and the middle and right panels compare the errors of the log determinant and the quadratic term respectively.}
    \label{fig:1d_re_along_trajectory}
\end{figure}

Next, we investigate the accuracy of the objective function along the optimization trajectory and provide relative errors in Figure~\ref{fig:1d_re_along_trajectory}. We first observe that the relative error in the objective function shows a steady increase until about $10$ iterations after which it plateaus. The computable bound tracks the error well until about $10$ iterations after which it shows poor quantitative behavior, although appears to be qualitatively good.  We plot in the middle and right panels the error in the log determinant and quadratic terms respectively, with computable bounds provided. We observe that the log-determinant term dominates the error and shows similar behavior as the overall error. Also, the bound for the log-determinant term is quantitatively informative, whereas the bound for the quadratic term is poor, which explains why the overall bound is poor close to the optimal solution. We conclude by emphasizing two observations: the overall error estimate is reasonable, and the bound for the log-determinant term is reasonable, so this approach can be used as an error indicator.

We also investigated the robustness of the optimization procedure by choosing $100$ different values of the initial guess $\vtheta_0$. The initial guesses are generated randomly by perturbing each coordinate uniformly at random by $50\%$ of the optimal parameter. Once the initial guesses are determined, the optimization problem is solved for both the ``exact'' case (using $\calF$) and the approximate case (using  $\widetilde{\calF}_k$) with the value $k = 22.$ We find that the optimization solutions converge to the same optimal solution with the relative error in the resulting images between $14-15\%$ thus demonstrating the robustness of the genGK approximations.

\paragraph{Experiment 3: Computational time.}
\begin{figure}[!ht]
    \begin{center}
    \includegraphics[scale=0.25]{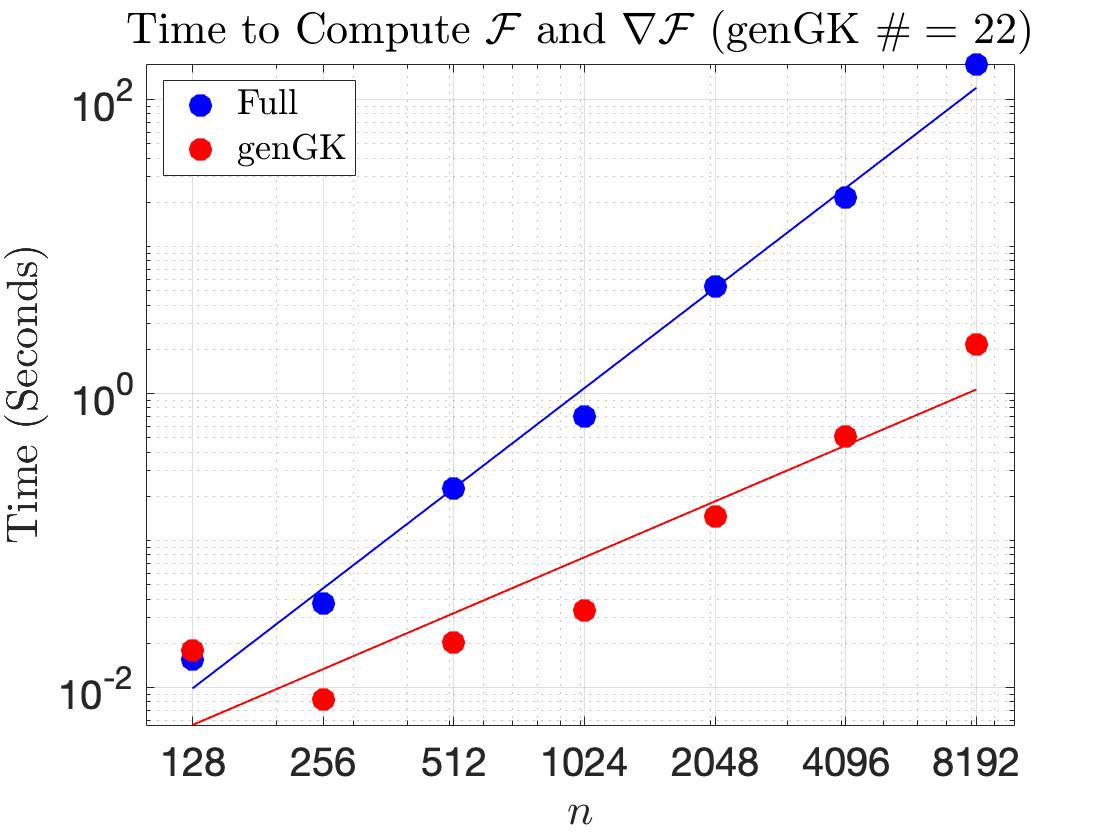}
    \caption{1D Heat: Wall clock time to compute the objective function and gradient pair versus the problem size $n$.}
    \label{fig:time_1d}
    \end{center}
\end{figure}

To demonstrate the computational efficiency of our approach, we compute the CPU time required to compute the objective function and the gradient with increasing numbers of unknowns $n$. In Figure~\ref{fig:time_1d}, we provide the wall clock time needed to compute the objective function-gradient pair for both the exact and approximate cases. For the case $n = 8192$, the genGK approximation provides an approximation faster than the ``full'' case by a factor of $81$.

\subsection{Application 2: Seismic tomography}\label{ssec:seismic}
This experiment corresponds to a two-dimensional problem in seismic travel-time tomography. Typically, this class of problems simulates geophysical situations in which measurements of the travel time of seismic waves are recorded between a collection of sources and detectors. Utilizing these measurements, one can reconstruct an image associated with the waves in some specified domain $\Omega = [0,1]^2$. We use the IR Tools~\cite{IR_Tools} package to generate the test problem. Using the \verb+[A, d, s, info] = PRseismic(n, options)+ command from the package, we obtain a forward operator, $\vA,$ a true solution, $\vs,$ and a right-hand-side of true observations, $\vd$. 
For the experiments below (via the input \verb+options+), we use the following settings: \verb|phantomImage = `smooth'|, \verb|wavemodel = `ray'|, \verb|s = 32|, and \verb|p = 45|.
This setup yields a forward operator with $1440$ measurements. 
As before, $2\%$ noise is added to the data to simulate measurement error.   Additionally, we choose the hyperprior (P2), $\pi \lrb{ \vtheta } \propto \exp\lrp{-\gamma \sum_{i=1}^3 \theta_i}$ with $\gamma = 10^{-4},$ and the covariance matrix is constructed using the Mat\'ern kernel with $\nu =3/2$. We begin with a similar set of experiments to the first application. We take the number of unknowns to be $4096$ since it is easier to compare the accuracy against the true objective function for a smaller problem. Contrary to application 1, this application is a significantly under-determined system.  Moreover, in Section \ref{sub:2param} we see how a 2-hyperparameter setting can be used for this problem.

\paragraph{Experiment 1: Accuracy of objective function.}
\begin{figure}[!ht]
    \centering
    \includegraphics[scale=0.21]{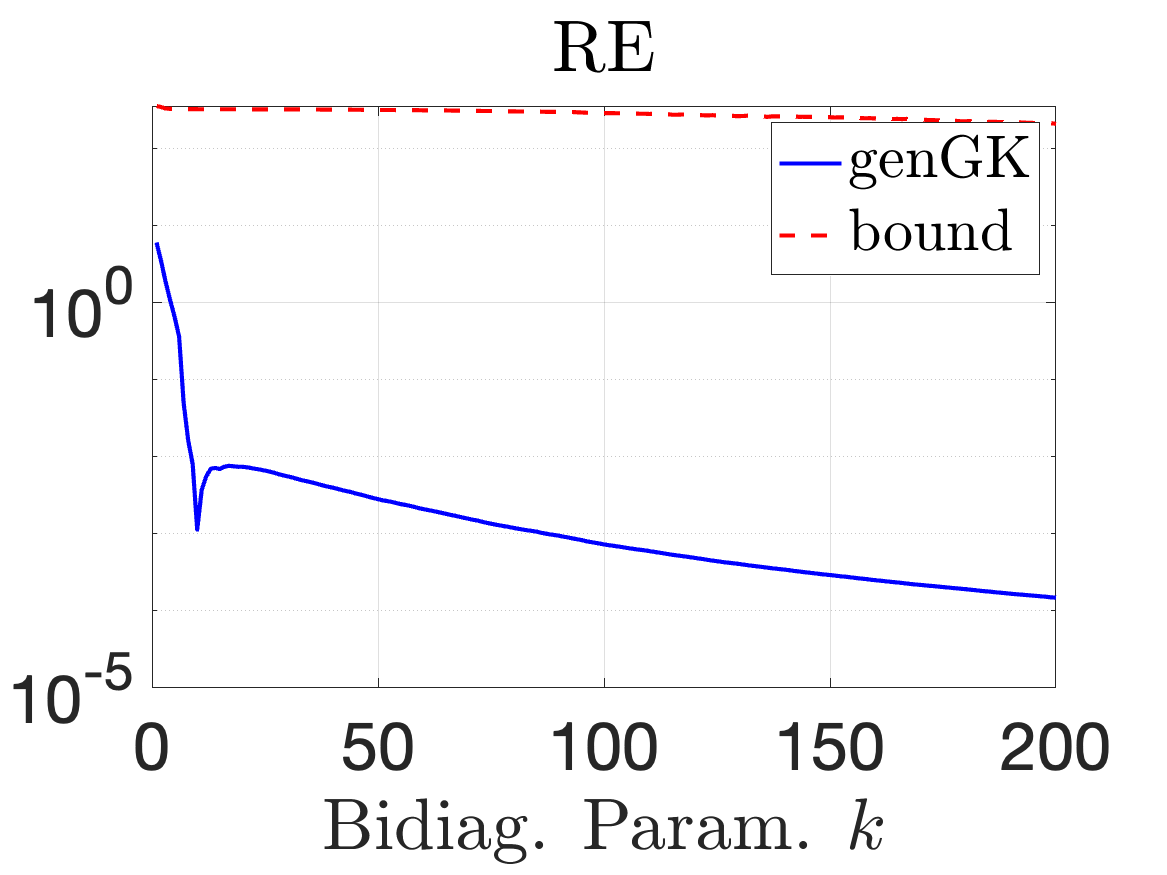}
    \includegraphics[scale=0.21]{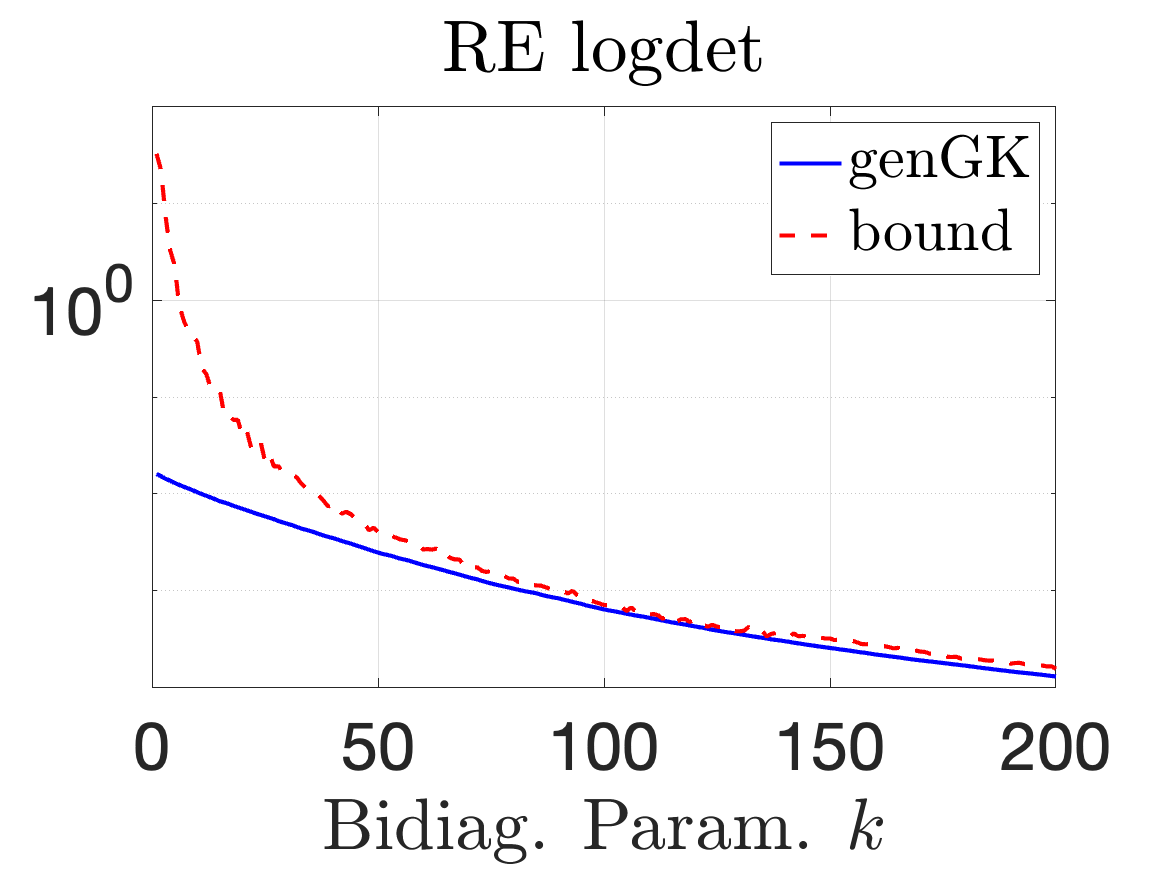}
    \includegraphics[scale=0.21]{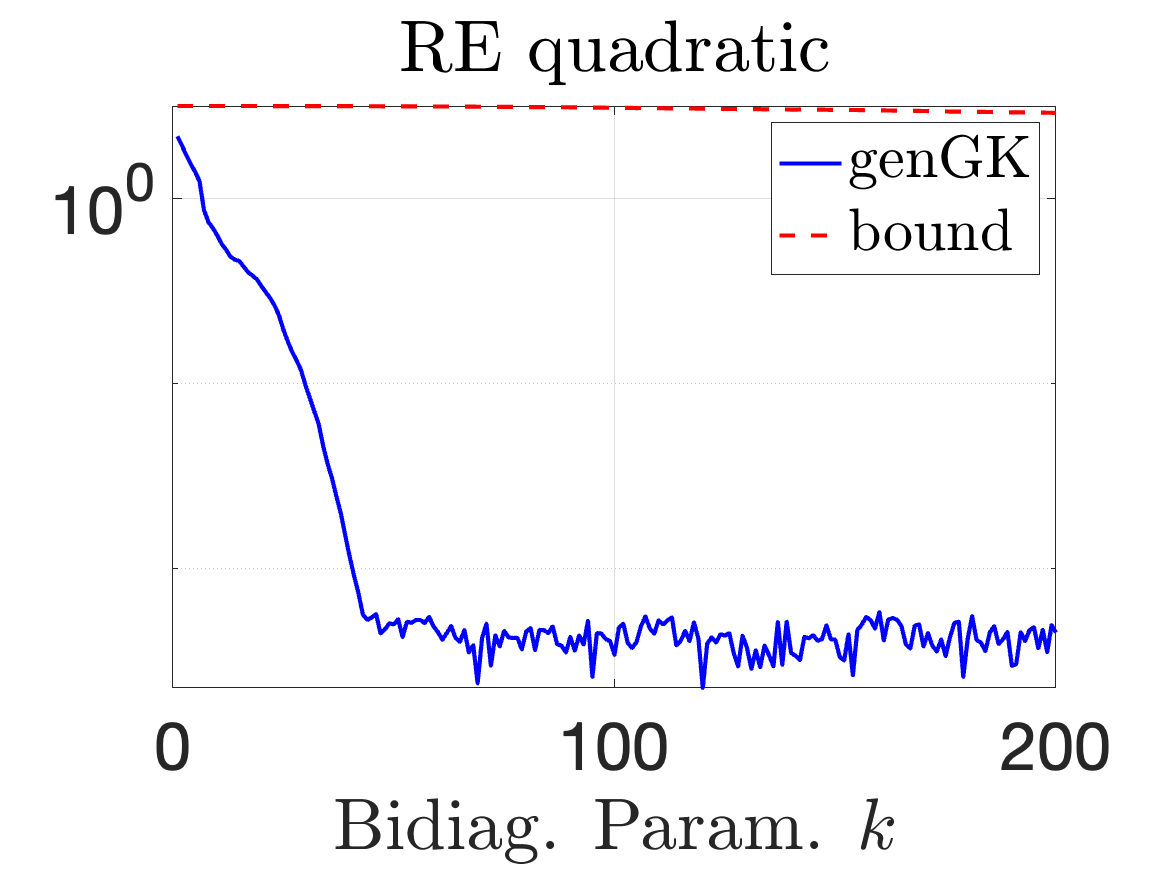}
    \caption{Seismic tomography: Relative errors and corresponding bounds for the objective function (left), log determinant term (middle), and quadratic term (right). }
    \label{fig::error_bounds_2D}
\end{figure}

\begin{figure}[!ht]
    \centering
    \includegraphics[scale=0.3]{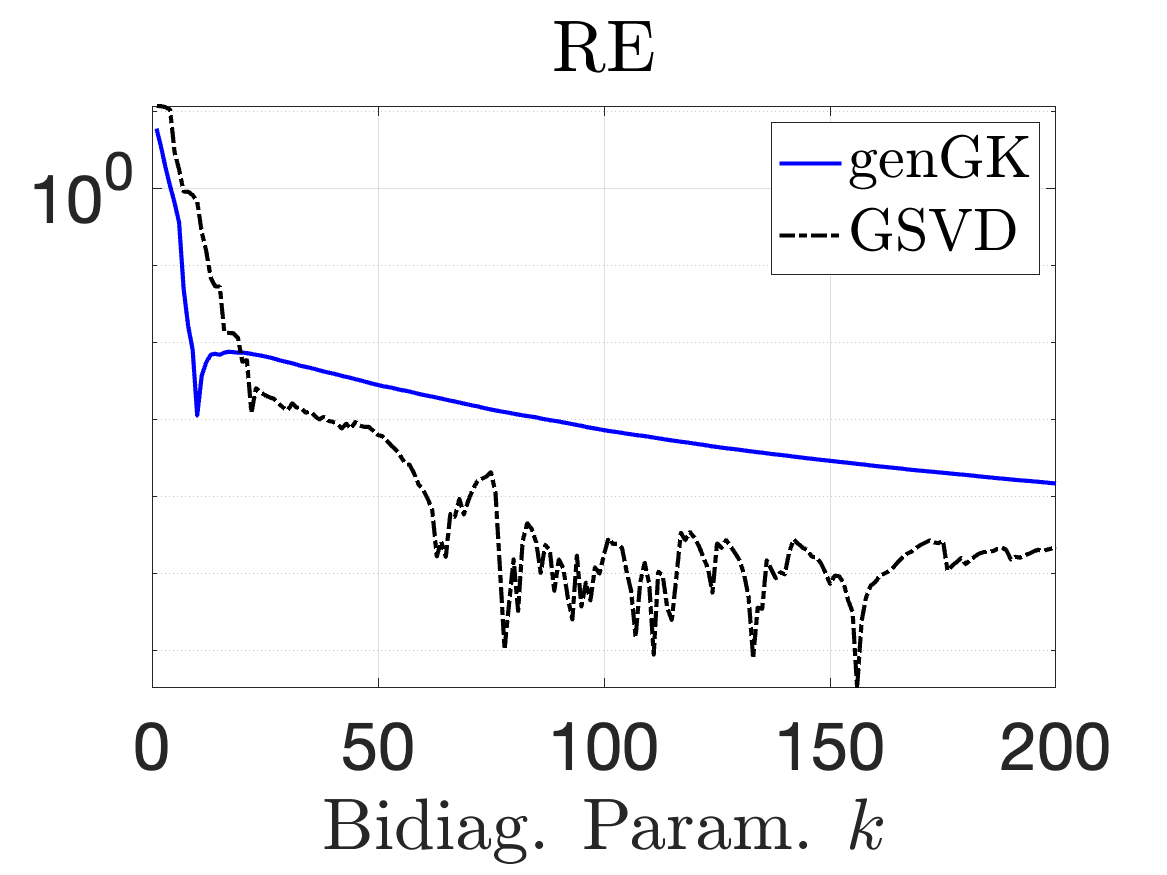}
    \qquad
    \includegraphics[scale=0.3]{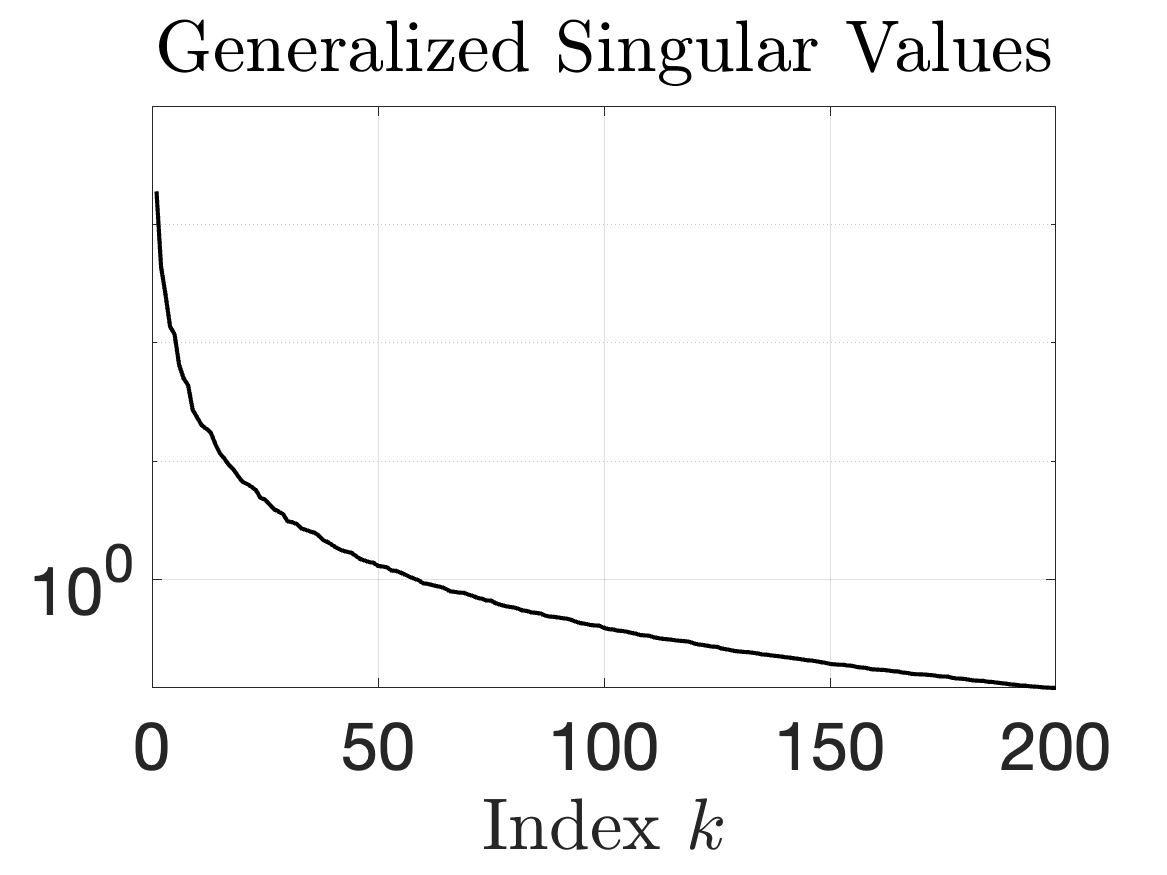}
    \caption{Seismic tomography: (left) Relative error in the objective function using genGK and GSVD, and (right) plot of singular values of $\widehat\vA$.}
    \label{fig::GenSingVals_2D}
\end{figure}

In Figure \ref{fig::error_bounds_2D}, we plot the relative error of the objective function at a point ${\vtheta}^*$ in blue and its estimated bound in red with $n_{mc}=100$. Similar to the 1D Heat example and as suggested by the decay of the singular values in Figure \ref{fig::GenSingVals_2D}, the relative errors exhibit a sharp decay with $k$; however, we see that the bound, while qualitatively good, is not accurate. By examining separately the two components, we observe that the bound for the log-determinant term is quite good while the bound for the quadratic term is not great. 
In subsequent numerical experiments, we used $k=200$ at which point the error in the objective function is about $10^{-5}$.

\paragraph{Experiment 2: Reconstruction and accuracy along optimal trajectories.}
The optimization procedure for estimating the hyperparameters took $23$ iterations and the number of objective function evaluations was $61$. The true and reconstructed solutions are provided in Figure~\ref{fig:2drecon}. The recovered image has about $3\%$ relative reconstruction error and similar qualitative features to the true solution. 

\begin{figure}[!ht]
    \centering
    \includegraphics[scale=0.15]{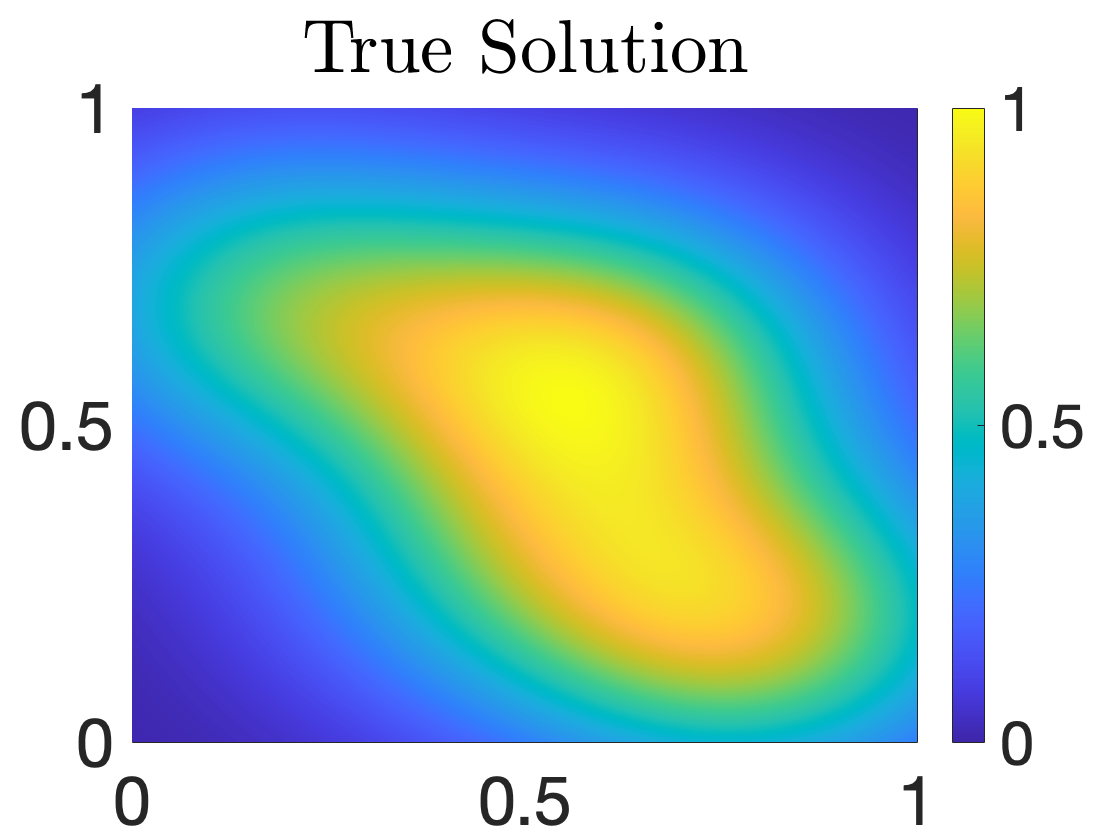}
    \includegraphics[scale=0.15]{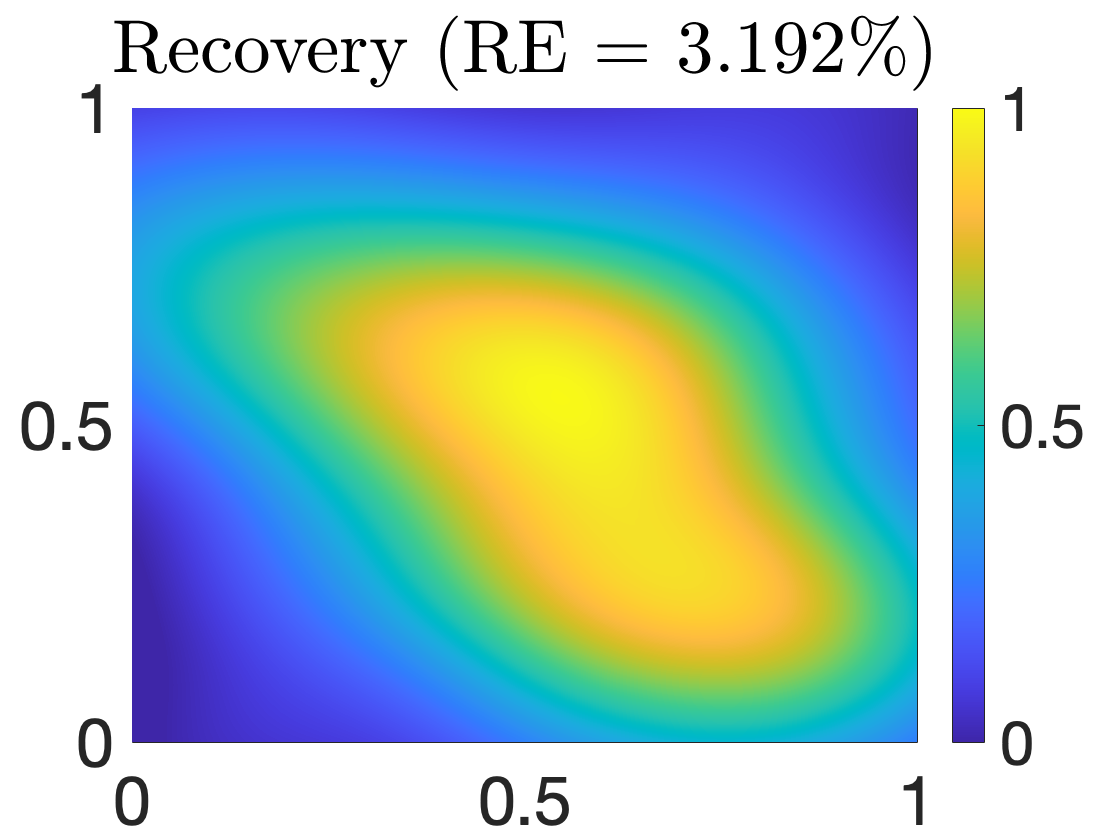}
    \caption{Seismic tomography: True solution and reconstruction obtained with computed hyperparameters. }
    \label{fig:2drecon}
\end{figure}

The previous experiment only considered approximations at a single set of parameter values for $\vtheta$. In Figure~\ref{fig:2dopttrajectory} we plot the relative errors of the objective function as well as the two components, log-determinant and quadratic, along the optimization trajectory.  As before, we see that the relative error is consistently good throughout the optimization trajectory. Looking more closely, we see that the relative error in the quadratic component is much smaller than the log determinant term (in fact, it is close to machine precision), suggesting that the genGK approximation is much better at approximating the quadratic term. Now considering the bounds, we see that the bound for the log determinant is very good both qualitatively and quantitatively, and as before the bound for the quadratic term is very poor. Therefore, we can use the bound for the log determinant as an indicator for the error in the objective function and the gradient. 

\begin{figure}[!ht]
    \centering
    \includegraphics[scale=0.21]{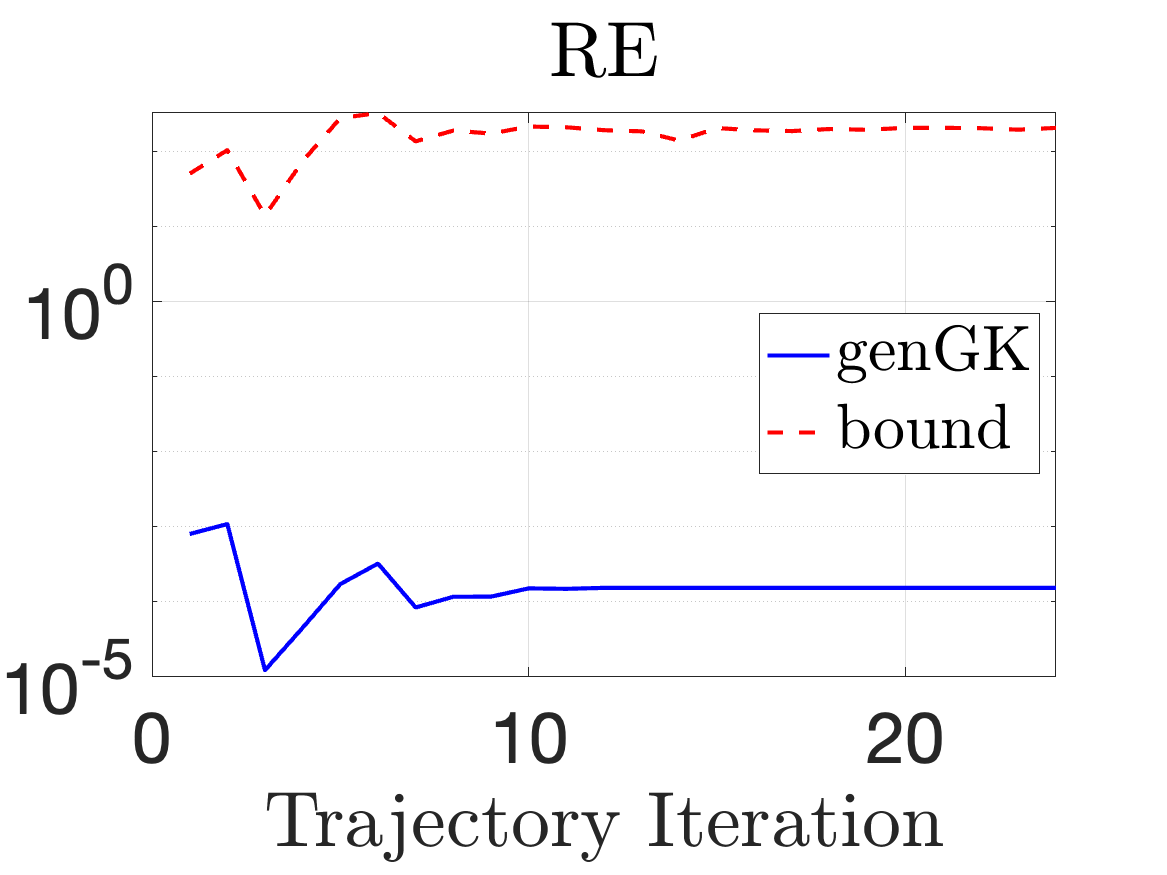}
    \includegraphics[scale=0.21]{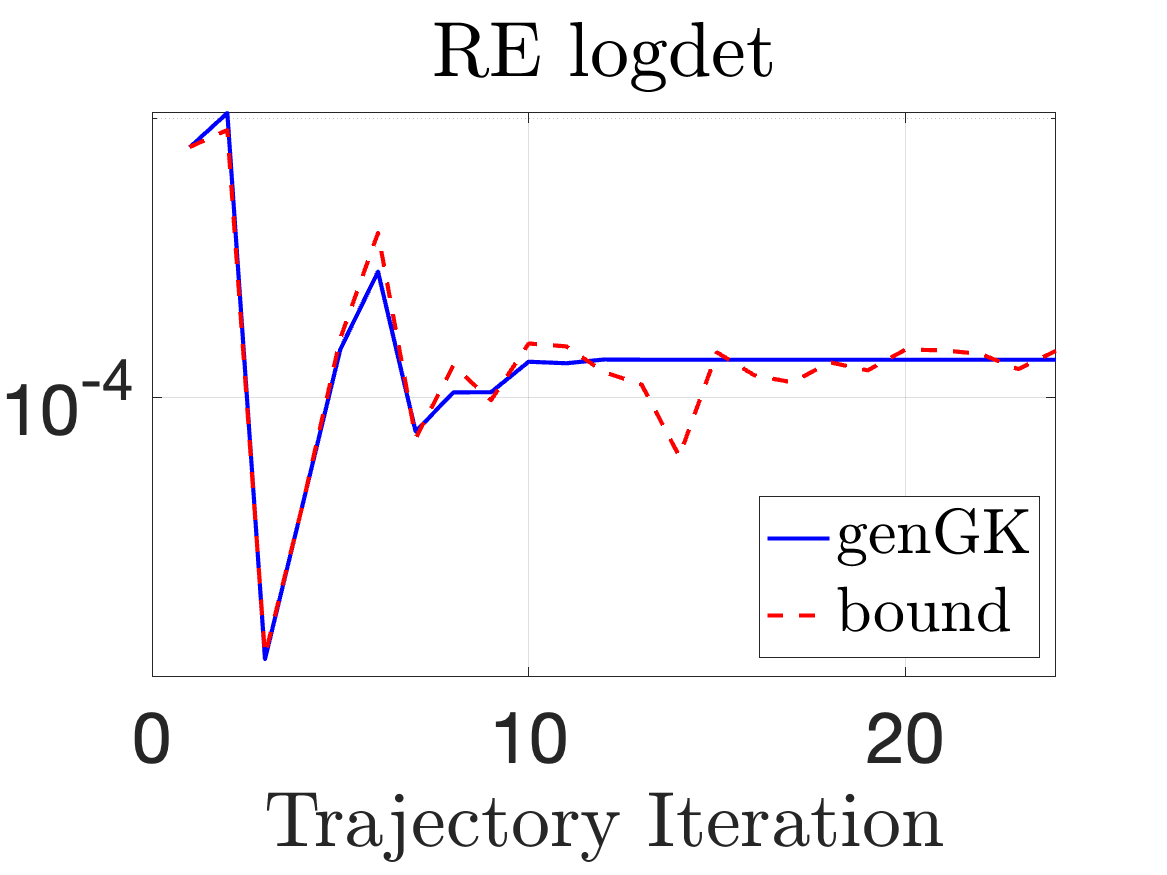}
    \includegraphics[scale=0.21]{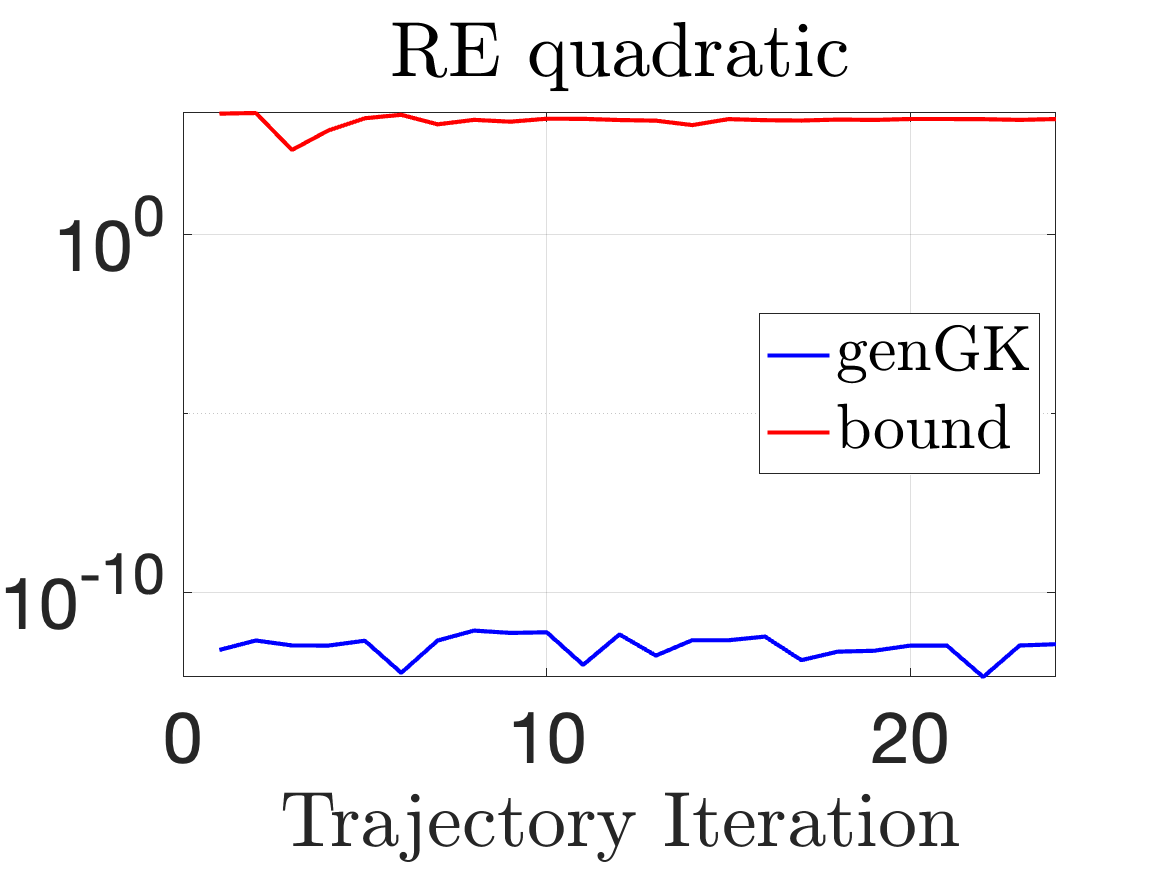}
    \caption{Seismic tomography: Relative errors and corresponding bounds along the optimal trajectory for $k=200$ and $n_{mc}=100$ samples. The left panel compares the overall accuracy of the objective function, and the middle and right panels compare the errors of the log determinant and the quadratic term respectively. }
     \label{fig:2dopttrajectory}
\end{figure}

\paragraph{Experiment 3: Computational time.}

In Figure~\ref{fig:time_2d} we compare the computational time of our approach against the full approach. We choose two values of $m$, namely $1440$ and $6400$ (with \verb|s=64| and \verb|p=100|). As before, we present the time to compute the objective function and the gradient for various numbers of unknowns $n$ from $2^8$ to $2^{16}$. We see that for both sets of measurements, the proposed approach that uses genGK (``Approximate'') is much faster than the full (``Exact'') approach. The difference is more pronounced for a larger number of measurements $m=6400$. 

These computations were performed on North Carolina State University's High Performance Computing cluster `Hazel' using MATLAB R2023a. The specifications are the same as those described in Experiment 3 of Section 4.1.
\begin{figure}[!ht]
    \begin{center}
    \includegraphics[scale=0.2]{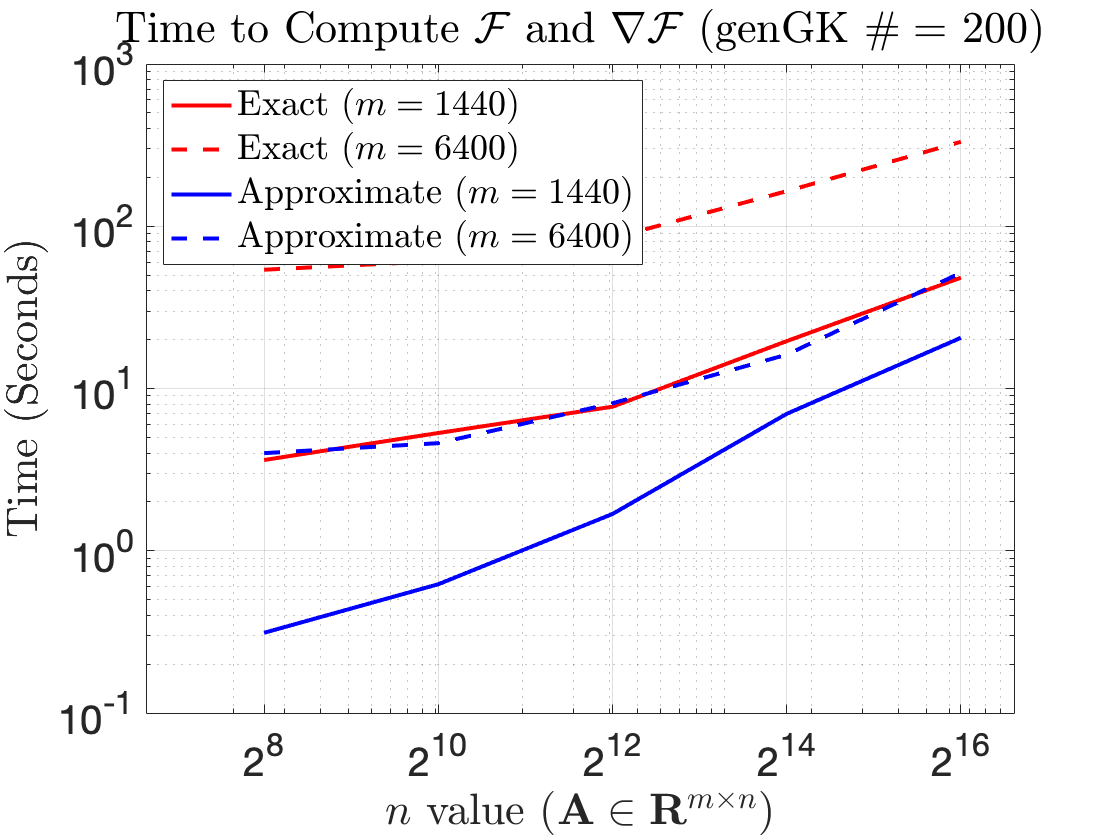}
    \caption{Seismic tomography: Wall clock time to compute the objective function and gradient for two different values of $m$: $1440$ and $6400$. }%
    \label{fig:time_2d}
    \end{center}
\end{figure}

\subsection{Two-Hyperparameter Setting}
\label{sub:2param}
In the previous examples, the parameter $\theta_1$ is related to the variance of the noise, the parameter $\theta_2$ is related to the variance of the prior, and $\theta_3$ is the correlation length of the prior covariance. Consider the setting where $\veta \sim \calN \lrp{\vzero, \theta_1 \Id_m },$ $\vs \sim \calN \lrp{ \vmu, \theta_2^2 \vQ_0 },$ and $\vQ_0 \in \R^{n \times n}.$ In this setting,  one has access to a reliable estimate for $\theta_3$, resulting in a fixed prior covariance matrix $\vQ_0$ and wishes to simply optimize over the parameters $\vtheta = \lrp{\theta_1, \theta_2} \in \R^2_{+}$. At first glance, it appears to be a special case of our approach, but this special case is of interest for two reasons: first, it allows us to compare against existing parameter selection methods that can recover a ratio of these parameters; and second, the computational cost of the optimization problem can be substantially lowered since the genGK relations can be precomputed in an offline stage. We elaborate on both points below.

To explain the first point, in~\cite{AKS_Chung_genHyBR} the genHyBR approach was considered for computing the MAP estimate
\[ \min_{\vs \in \mathbb{R}^n} \|\vd - \vA\vs\|_{\vR^{-1}}^2 + \lambda^2 \| \vs - \vmu\|^2_{\vQ_0^{-1}}, \]
with methods to automatically estimate the regularization parameter $\lambda$, e.g., using the Generalized Cross Validation (GCV) method, the Unbiased Predictive Risk Estimate (UPRE), the Discrepancy Principle, etc. The connection between this formulation and the present approach is that $\lambda^2 = 1/\theta_2^2$. A key distinction of our approach presented here is that we can recover both $\theta_1$ and $\theta_2$, whereas genHyBR can recover only the ratio $\sqrt{\theta_1}/\theta_2$ (assuming that $\theta_1$ is known a priori). We will present a comparison with genHyBR below.

To expand on the second point, we argue that the genGK process needs to be run only once, and the computed matrices can be reused during the optimization process. That is, we run Algorithm~\ref{algorithm:genGK} as genGK($\vA,\vI,\vQ_0,\vmu,\vd,k$) to obtain the matrices $\widehat{\vU}_{k+1}, \widehat{\vB}_{k}$, and $\widehat{\vV}_k$ that satisfy $\vA \approx \widehat\vU_{k+1} \widehat\vB_{k} \widehat\vV_k^{\top}$ with the orthogonality relations $\widehat\vU_{k+1}^{\top}  \vU_{k+1} = \Id_{k+1}$ and $\widehat\vV^{\top}_{k} \vQ_0  \widehat\vV_{k} = \Id_{k}.$ To obtain the genGK relations for any particular value of $\vtheta = \lrp{\theta_1, \theta_2}$ one simply has to perform the change of variables 
\[ \vU_{k+1} =\sqrt{\theta_1} \widehat\vU_{k+1}, \qquad \vV_k = \theta_2^{-1} \widehat\vV_k, \qquad \vB_k = \frac{\theta_2}{\sqrt{\theta_1}} \widehat\vB_k. \]
This gives the genGK iterations at the new optimization point that satisfies $\vA \approx \widehat\vU_{k+1} \widehat\vB_{k} \widehat\vV_k^{\top}= \vU_{k+1} \vB_{k} \vV_k^{\top}$ with the new orthogonality relations 
\[\vU_{k+1}^{\top} (\theta_1\vI_m)^{-1}  \vU_{k+1} = \Id_{k+1} \qquad\vV^{\top}_{k} (\theta_2^2 \vQ_0) \vV_{k} = \Id_{k}.\] 
These new genGK relations can be used in conjunction with the approximations developed in Section~\ref{ssec:gengk}. More specifically, the computational cost at each iteration to determine the objective function and the gradient is $\mathcal{O}(k(m+n) + k^3)$ flops. The important point is that each iteration does not require applying the forward operator, and the optimization procedure is very efficient since the (expensive) genGK step is precomputed.

\paragraph{Application to Seismic Tomography.} Consider the seismic tomography problem (Section~\ref{ssec:seismic}) with the same settings except we take $\vA \in \R^{1440 \times 65536}$ and the number of genGK iterations to be $150$. The value of $\theta_3$ was chosen based on the optimal value from the previous experiments and was set to $\theta_3^* = 0.90216$. Using our approach we obtained the optimal value  
$\vtheta_{*} = \lrp{1.07\times 10^{-4}, 0.49}^{\top}$. 
  Additionally, in Table \ref{tab:2-hyper}, we provide the values of the regularization parameter {computed using genHyBR for comparison~\cite{AKS_Chung_genHyBR}. Since genHyBR can only estimate the parameter $\lambda$, rather than $\theta_1$ and $\theta_2$ separately, we fixed the value of the noise variance to be $\theta_1 = 1.07\times 10^{-4}$.
 Then, we compute at each iteration the optimal regularization parameter $\lambda_{\rm opt}$ and denoted ``Optimal'', which minimizes the 2-norm of the error between the reconstruction and the truth. We also consider the parameter selection techniques using the weighted generalized cross-validation approach (WGCV) and the discrepancy principle (DP). The details are given in~\cite{AKS_Chung_genHyBR}. }

\begin{table}[h!] \centering 
\begin{tabular}{ c|c|c|c|c} 
 & Our Approach  &  ``Optimal" & WGCV & DP \\ 
 \hline
 $\theta_2$ or $1/\lambda$ & $0.49$& $0.46$ & $0.09$ & $0.27$ \\ 
 \hline
 Rel. Err. & $2.48\%$ & $2.44\%$ & $4.59\%$ & $3.41\%$ 
\\ \hline
Iterations (MAP estimate) & $150$ & $43$ &  $58$ & $18$
\end{tabular}
\caption{Comparison of the regularization parameter $\lambda$ or $1/{\theta_2}$ after optimizing over a two-dimensional hyperparameter space.}
\label{tab:2-hyper}
\end{table}
Table~\ref{tab:2-hyper} lists the value of $\theta_2$ (which is also the inverse regularization parameter) obtained using the different techniques. Our approach which optimizes over $\theta_1$ and $\theta_2$ has a comparable relative error to the ``Optimal'' approach which assumes knowledge of $\theta_1$ and the true solution. Our approach also does better than WGCV and DP approaches in terms of relative error. In summary, the proposed approach is more general and compares favorably against existing special-purpose parameter selection techniques.

\subsection{Atmospheric tomography}
In this last numerical experiment, we illustrate our methodology on an atmospheric inverse modeling example \cite{chung2022hybrid}, where the data vector $\vd$ represents noisy satellite observations of a hypothetical atmospheric trace gas. The matrix $\vA$ represents a forward atmospheric model that will transport this gas based on estimated atmospheric winds and simulate satellite observations. Specifically, $\vd$ is a vector of samples at the locations and times of CO$_2$ observations from NASA's Orbiting Carbon Observatory 2 (OCO-2) satellite during July through mid-August 2015, and $\vA$ is generated from atmospheric model runs conducted as part of NOAA's CarbonTracker-Lagrange Project \cite{miller2020,noaa2023}. 
\begin{figure}[!ht]
    \begin{center}
    \subfloat[\centering True Emissions]
    {{\includegraphics[width=6cm]{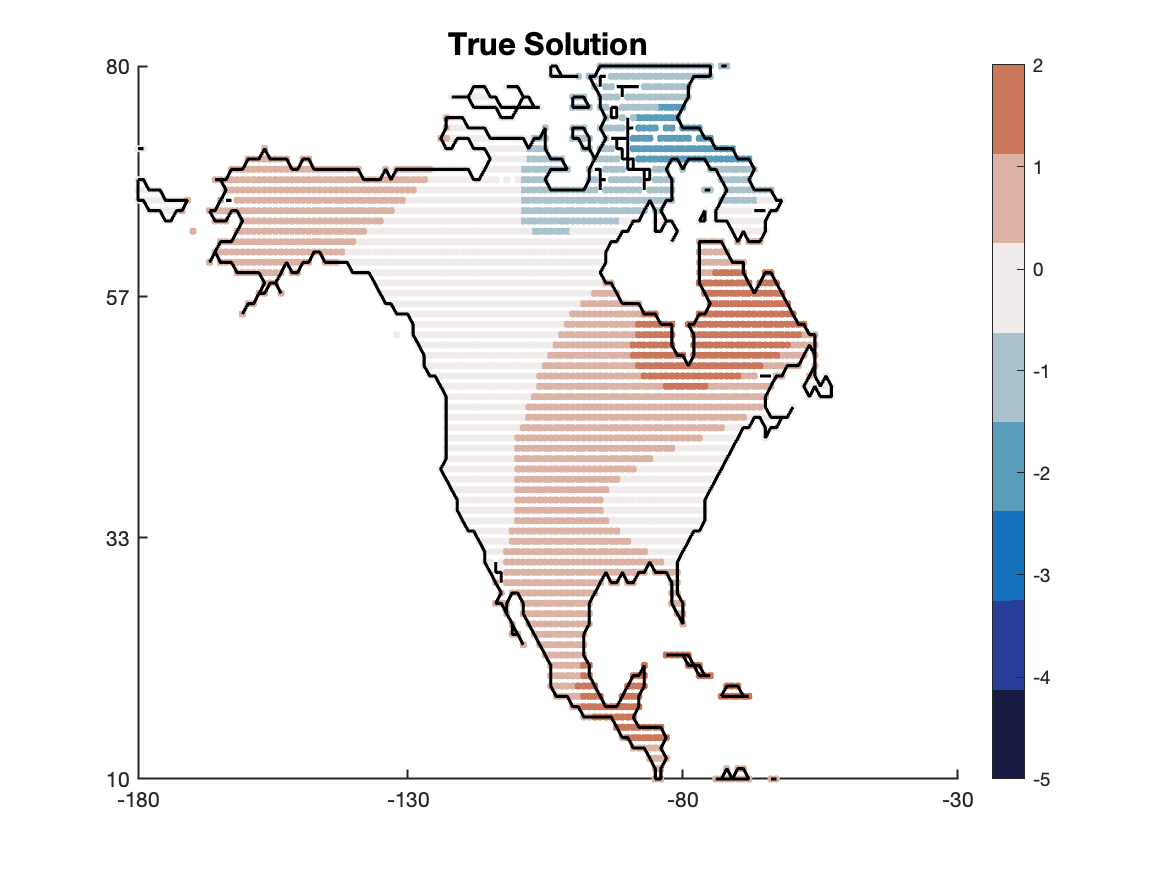} }}
    \qquad
    \subfloat[\centering Reconstruction]
    {{\includegraphics[width=6cm]{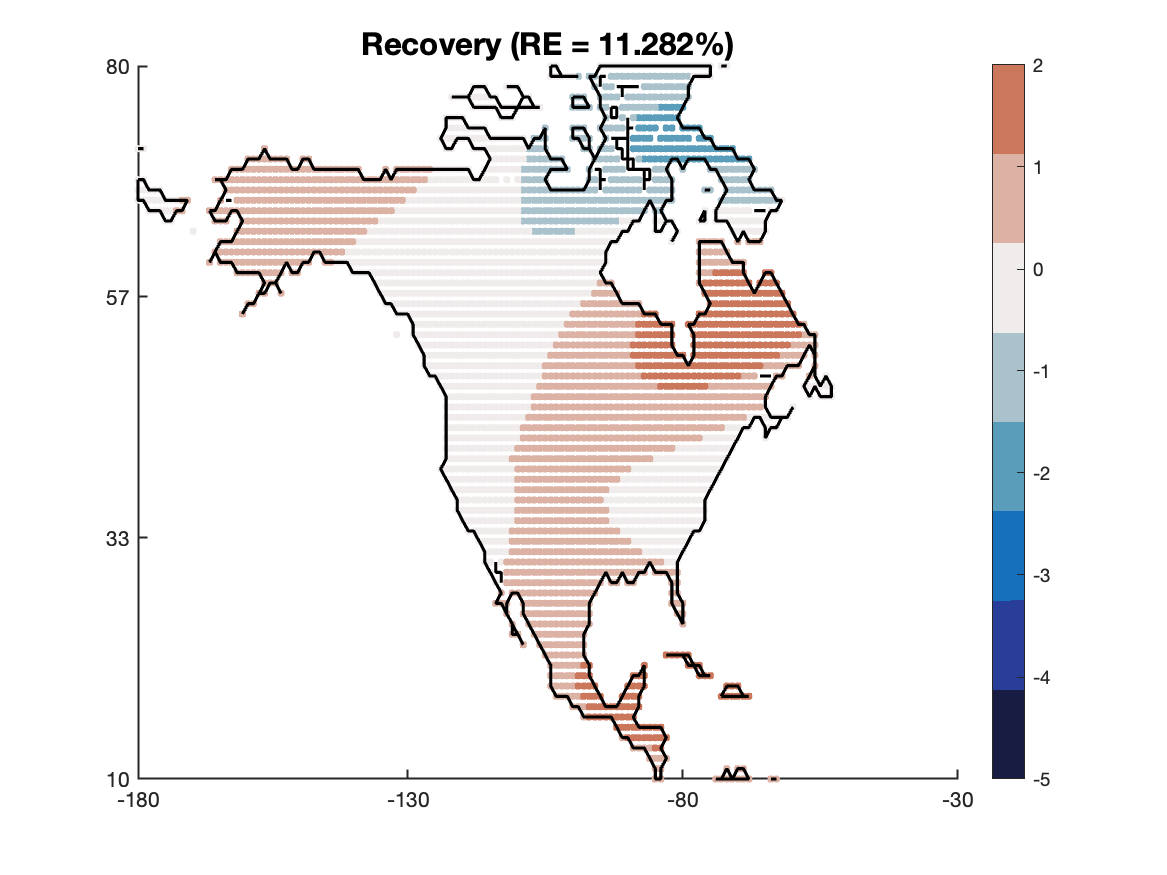} }}%
    \end{center}
    \caption{Atmospheric tomography: (left) True synthetic emissions pattern, and (right) reconstruction with estimated hyperparameters. }
    \label{fig:2d_atmos}
\end{figure}

We generate a synthetic atmospheric transport problem where 
$\vA \in \R^{98880 \times 3222}$ and $\vd \in \R^{98880}$, and the goal is to reconstruct the unknown set of states or trace gas fluxes across North America, where the spatial resolution is $1^{\circ} \times 1^{\circ}.$ Although the unknown parameters are over land, to exploit the advantages of the fast Fourier transform, we treat the unknowns on a regular grid and then incorporate a masking operator in the definition of $\vA$. More specifically, the ``unknown" vector $\vs$ is provided in Figure \ref{fig:2d_atmos}(a) and is an image generated using a truncated Karhunen-Lo\'eve expansion generated with a Mat\'{e}rn kernel with parameters $\nu = 2.5$ and $\ell = 0.05$ followed by a land mask.  The observed data $\vd$ were obtained by applying the atmospheric forward model $\vA$ and corrupting it with $2\%$ noise to simulate measurement noise. For additional details, we refer the reader to \cite{miller2020,cho2022,chung2022hybrid}.

For the hyperparameter estimation and reconstruction, we take $\vQ$ defined by a Mat\'ern kernel with $\nu=3/2$. For the hyperprior, we used a noninformative prior (P1), and for genGK, we used $k=250$ iterations. For the initial guess, we used an estimate of the noise variance $\theta_1$ and the length scale $\theta_3 = 0.075$ (the length scale is based on a domain of length $1$; we used a random guess for $\theta_2$. The optimizer took $19$ iterations and $94$ number of function evaluations.  The optimal value of $\vtheta$ was found to be 
\[\vtheta^* = \lrp{8.51\times 10^{-5}, 1.068, 0.0135}.\]
The reconstruction is provided in Figure~\ref{fig:2d_atmos}(b), where the resulting relative reconstruction error norm was approximately $11.2\%$. The optimization procedure took about $54$ minutes in wall clock time.

\section{Conclusions}
\label{sec:conclusions}
This paper describes an efficient approach for hyperparameter estimation that combines an empirical Bayes approach with the generalized Golub-Kahan bidiagonalization process for large-scale linear inverse problems. We consider the marginal posterior distribution and derive efficient algorithms for computing the MAP estimate of the marginalized posterior.  We derive an approximation to the objective function and its gradient, where the approximation is computed using information from the genGK bidiagonalization, and several results quantifying the accuracy of the approximations are provided. We also have a method to estimate and track the error in the approximations. We demonstrated the performance of our approach on model problems from inverse heat equation (1D), seismic and atmospheric tomography (2D). While our approach is fairly general, in numerical experiments we only considered the Mat\'ern covariance models with a small number of hyperparameters. Similarly, while we only considered the empirical Bayes method here, the approximations that we derive are useful in other formulations such as fully Bayesian and variational Bayes.

In future work, we will consider more robust methods for estimating the hyperparameters that do not rely solely on low-rank approximations of $\vA$. Another aspect worth exploring is whether we can improve the bound for the quadratic term. We will also consider extensions to more realistic time-dependent inverse problems, such as from atmospheric and dynamic tomography problems, where low-rank structures or sparsity patterns in the temporal direction may be exploited.  Such problems come with additional computational challenges, including more hyperparameters to estimate and larger numbers of unknowns (e.g., corresponding to finer spatiotemporal discretizations).

\paragraph{Acknowledgements.}
This work was partially supported by the National Science Foundation ATD program under grants DMS-2026841/2341843, 2026830, 2026835, and 1845406. Any opinions, findings, and conclusions or recommendations expressed in this material are those of the author(s) and do not necessarily reflect the views of the
National Science Foundation. 

We acknowledge the computing resources provided by North Carolina State University High-Performance Computing Services Core Facility (RRID:SCR\_022168). We also thank Andrew Petersen (aapeters@ncsu.edu) for his assistance.

\section*{Declarations} 
The authors have no relevant financial or non-financial interests to disclose.

\appendix

\section{Details of a derivation}\label{sec:derivations}

We show how to derive~\eqref{eqn:term1simp}. Using the Sherman--Morrison--Woodbury (SMW) formula, we have $\tilde{\vZ}^{-1}_{k} = (\vU_{k+1} \vB_{k} \vB^{\top}_{k} \vU^{\top}_{k+1} + \vR)^{-1} = \vR^{-1} - \vR^{-1} \vU_{k+1} \vB_{k} \left(\Id_{k} + \vB^{\top}_{k} \vB_{k} \right)^{-1} \vB^{\top}_{k} \vU^{\top}_{k+1} \vR^{-1}.$ 

Plugging in the expression for $\tilde{\vZ}^{-1}_k$ and using the linearity of trace
\begin{equation*}
    \begin{aligned}
    \inn{\tilde{\vZ}^{-1}_{k}}{ \widetilde{\partial_{\theta_i} \vZ}_{k}}_{F} &= \trace \left( \tilde{\vZ}^{-1}_k \widetilde{ \partial_{\theta_i} \vZ}_{k} \right) \\
        &= \trace \left( \left[ \vR^{-1} - \vR^{-1} \vU_{k+1} \vB_{k} \left(\Id_{k} + \vT_{k} \right)^{-1} \vB^{\top}_{k} \vU^{\top}_{k+1} \vR^{-1} \right] \left[ \widetilde{\vA}_k {\partial_{\theta_i} \vQ} \widetilde{\vA}^{\top}_k + {\partial_{\theta_i} \vR} \right] \right) \\
        &= \trace\left( \vR^{-1} \widetilde{\vA}_k \partial_{\theta_i} \vQ \widetilde{\vA}^{\top}_k \right) + \trace \left( \vR^{-1} \partial_{\theta_i}\vR \right) \\ 
        & \hspace{.45cm} - \trace \left( \vR^{-1} \vU_{k+1} \vB_{k} \left(\Id_{k} + \vT_{k} \right)^{-1} \vB^{\top}_{k} \vU^{\top}_{k+1} \vR^{-1} \widetilde{\vA}_k \partial_{\theta_i} \vQ  \widetilde{\vA}^{\top}_k \right) \\
        & \hspace{.45cm} - \trace \left( \vR^{-1} \vU_{k+1} \vB_{k} \left(\Id_{k} + \vT_{k} \right)^{-1} \vB^{\top}_{k} \vU^{\top}_{k+1} \vR^{-1} \frac{\partial \vR(\bar{\vtheta})}{\partial \theta_{i}} \right) .
        \end{aligned}
\end{equation*}
We plug in the expansion for $\widetilde{\vA}_k = \vU_{k+1} \vB_{k} \vV^{\top}_{k}$, use the cyclic property of the trace, and identify the matrices $\vPsi_i^Q$ and $\vPsi_i^R$ to get 
\begin{equation*}
  \begin{aligned}
 \inn{\tilde{\vZ}^{-1}_{k}}{ \widetilde{\partial_{\theta_i} \vZ}_{k}}_{F} &= \trace\left(  \vT_{k} \vPsi_i^Q   \right) + \trace \left( \vR^{-1}  \partial_{\theta_i}\vR \right) \\ 
        & \hspace{.45cm} - \trace \left( \left(\Id_{k} + \vT_{k} \right)^{-1} \vT_{k} \vPsi_{i}^Q \vT_{k} \right)  - \trace \left( \left(\Id_{k} + \vT_{k} \right)^{-1} \vB^{\top}_{k} \vPsi_i^R \vB_{k} \right). 
    \end{aligned}
\end{equation*}
In deriving these expressions we have also used the orthogonality relations $\vU^{\top}_{k+1} \vR^{-1} \vU_{k+1} = \Id_{k+1}.$ Finally, we can use $\vI - (\vI_k + \vT_k)^{-1} \vT_k = (\vI_k + \vT_k)^{-1}$ to obtain the simplified expression~\eqref{eqn:term1simp}.

\section{Error analysis in the objective function}
\subsection{Proof of Proposition~\ref{prop::svd_bound}} \label{ssec:svd}
Consider the absolute error
\begin{align}\label{eqn:errorgsvd}
    \bar{\calF} - \widehat{\calF}_{k} = \frac{1}{2} T_1 + \frac{1}{2} T_2,
\end{align}
where $T_1 \equiv |\logdet(\vZ) - \logdet({\vZ}_{k})|$ and $T_2 \equiv \| \vA \vmu - \vd \|^{2}_{\vZ^{-1}} - \| \vA \vmu - \vd \|^{2}_{{\vZ}^{-1}_{k}}.$ We bound each of these terms separately. 

For $T_1$, first note
\[T_1 = |\logdet(\vZ) - \logdet({\vZ}_{k})| = | \logdet(\vI_m + \widehat\vA\widehat\vA^\top ) -\logdet(\vI_m + \widehat\vA_k\widehat\vA_k^\top ) .  \]
Next notice that $\widehat\vA\widehat\vA^\top \succeq \widehat\vA_k\widehat\vA_k^\top $; apply both parts of~\cite[Lemma 9]{alexanderian2018efficient} to get 
\[ T_1 = \logdet(\vI_m + \widehat\vA\widehat\vA^\top -  ) -\logdet(\vI_m + \widehat\vA_k\widehat\vA_k^\top ) \leq \logdet(\vI_m +  \widehat\vA\widehat\vA^\top -\widehat\vA_k\widehat\vA_k^\top ) .\]
Plug in the SVD of $\widehat\vA$ and simplify to obtain 
$T_1 \leq \sum_{j > k} \log(1 +\sigma_j^2(\widehat\vA))$. 

For $T_2$, let $\vr =\vR^{-1/2}(\vA \vmu - \vd)$ so that 
\[ \begin{aligned}T_2 = & \>  \| \vA \vmu - \vd \|^{2}_{\vZ^{-1}} - \| \vA \vmu - \vd \|^{2}_{{\vZ}^{-1}_{k}}  \\
= & \> \vr^\top ((\vI_m + \widehat\vA\widehat\vA)^{-1} - (\vI_m +\widehat\vA_k\widehat\vA_k^\top)^{-1} )\vr. \end{aligned}\] 
Applying Cauchy-Schwartz inequality followed by submultiplicativity, we get $$T_2 \leq \beta_1^2 \|(\vI_m + \widehat\vA\widehat\vA)^{-1} - (\vI_m+\widehat\vA_k\widehat\vA_k^\top)^{-1}\|_2. $$
Apply~\cite[Lemma X.1.4]{Bhatia_Matrix_Analysis} and simplify to get $T_2 \leq \beta_1^2 \sigma_{k+1}^2(\widehat\vA)/(1+\sigma_{k+1}^2(\widehat\vA))$.

The proof is complete by plugging in the bounds for $T_1$ and $T_2$ into~\eqref{eqn:errorgsvd}.
\subsection{Proof of Proposition~\ref{prop:genGK_bound}} \label{ssec:proofgengk}
As in the proof of Proposition~\ref{prop::svd_bound}, write 
\begin{align}
    \bar{\calF} - \tilde{\calF}_{k} = \frac{1}{2} T_1 + \frac{1}{2} T_2,
\end{align}
where $T_1 \equiv \logdet(\vZ) - \logdet(\tilde{\vZ}_{k})$ and $T_2 \equiv \| \vA \vmu - \vd \|^{2}_{\vZ^{-1}} - \| \vA \vmu - \vd \|^{2}_{\tilde{\vZ}^{-1}_{k}}.$ We bound each of these terms separately. 

Let us first consider $T_1.$ Set $\widehat{\vA} = \vR^{-1/2} \vA \vQ^{1/2}.$ By Sylvester's determinant identity,
\begin{equation}
    \logdet(\vZ) = \logdet(\vR) + \logdet(\Id_{m} + \widehat{\vA} \widehat{\vA}^{\top}).
\end{equation}
We can write $\tilde{\vA}_{k} \vQ \tilde{\vA}^{\top}_{k}$ in terms of the projector $\vPi_{\vV_k}$ as 
\[
    \tilde{\vA}_{k} \vQ \tilde{\vA}^{\top}_{k} = \vA\vP_{\vQ}^\top \vQ (\vA\vP_{\vQ}^\top)^\top     =    \vA \vQ^{1/2} \vPi_{\vV_{k}}  (\vA \vQ^{1/2})^{\top},
\]
we can assert that $\tilde{\vZ}_{k} = \vR^{1/2} \left[ \widehat{\vA} \vPi_{\vV_{k}} \widehat{\vA}^{\top} + \Id_{m} \right] \vR^{1/2}.$ Hence,
\begin{equation}
    \logdet(\tilde{\vZ}_{k}) = \logdet(\vR) + \logdet(\Id_{m} + \widehat{\vA} \vPi_{\vV_{k}} \hat{\vA}^{\top})
\end{equation}
Set $\vM = \widehat{\vA} \widehat{\vA}^{\top}$ and $\vN = \widehat{\vA} \vPi_{\vV_{k}} \widehat{\vA}^{\top}.$ Since the orthogonal projector $\vPi_{\vV_{k}} \preceq \Id_n$ then $\vN \preceq \vM$ and, by~\cite[Lemma 9]{alexanderian2018efficient}
\begin{align*}
0 \leq T_1 &\leq \logdet(\Id_{m} + \vM - \vN) \\
&\leq \trace(\vM - \vN)  =  \trace(\vH_{\vQ}) - \trace(\widehat{\vH}^{(k)}_{\vQ}) \equiv \xi_{k},
\end{align*}
where in the second step we have used $\logdet(\vX) \leq \trace(\vX - \Id_{m})$ for any positive semidefinite matrix $\vX \in \calM_{m}(\R)$ and in the last step we have used the cyclic and linearity properties of the trace operator.

Next, consider $T_2.$ Set $\vr = \vR^{-1/2}(\vA \vmu - \vd)$ and recall that $\| \vr \|_{2} = \beta_1$ in (\ref{genGK_relations_transformed}). We can rewrite 
\[\| \vA \vmu - \vd \|^{2}_{\vZ^{-1}} = \left\langle (\Id_m + \widehat{\vA} \widehat{\vA}^{\top})^{-1} \vr  , \vr\right\rangle_{2} = \left\langle f(\widehat{\vA}\widehat{\vA}^\top) \vr  , \vr\right\rangle_{2}, \]
where $f : \R_{+} \rightarrow \R : x \mapsto 1 - (1 + x)^{-1}$. 
Then using the Cauchy-Schwartz inequality 
\begin{align*}
    T_2 =& \> \left\langle f( \widehat{\vA} \widehat{\vA}^{\top}) - f( \widehat{\vA}\vPi_{\vV_{k}}  \widehat{\vA}^{\top}) )\vr, \vr \right\rangle_{2} \\
     \leq & \> \norm{ f( \widehat{\vA} \widehat{\vA}^{\top}) - f( \widehat{\vA}\vPi_{\vV_{k}}  \widehat{\vA}^{\top}) ) } \beta^{2}_1.
\end{align*}
Since $f(x)$ is operator monotone (see, for example,~\cite[Exercise V.1.10 (ii)]{Bhatia_Matrix_Analysis}) and $f(0) = 0.$ Hence, by~\cite[Theorem X.1.1]{Bhatia_Matrix_Analysis},
\[\begin{aligned}
    T_2 &\leq \norm{f(\widehat{\vA} \widehat{\vA}^{\top}) - f(\widehat{\vA} \vPi_{\vV_{k}}\widehat{\vA}^{\top}) }_{2} \beta^{2}_1 \\
    &\leq \beta^{2}_1  f\left(\norm{\widehat{\vA} \left[ \Id_{n} -  \vPi_{\vV_{k}} \right] \widehat{\vA}^{\top}}_{2}\right) 
    =  \frac{\beta^{2}_1  \norm{\widehat{\vA} \left[ \Id_{n} -  \vPi_{\vV_{k}} \right] \widehat{\vA}^{\top}}_{2}}{1 + \norm{\widehat{\vA} \left[ \Id_{n} -  \vPi_{\vV_{k}} \right] \widehat{\vA}^{\top}}_{2}}  .
\end{aligned}\]
Since ${\Id_{n} -  \vPi_{\vV_{k}}}$ is an orthogonal projector, it is idempotent and its $2$-norm is at most $1$; using these two facts
\begin{align*}
    \norm{\widehat{\vA} \left[ \Id_{n} -  \vPi_{\vV_{k}} \right] \widehat{\vA}^{\top}}_{2} &= \norm{ \left[ \Id_{n} -  \vPi_{\vV_{k}} \right] \widehat{\vA}^{\top} \widehat{\vA} \left[ \Id_{n} -  \vPi_{\vV_{k}} \right]}_{2} \\
    &\leq \norm{ \left[ \Id_{n} -  \vPi_{\vV_{k}} \right] \vH_{\vQ}}_{2} \leq |\trace( [\Id_{n} -  \vPi_{\vV_{k}} ] \vH_{\vQ})| =\xi_k.
\end{align*}
Hence, $T_2 \leq \frac{\beta^{2}_{1} \xi_{k}}{1 + \xi_{k}}.$ Combine the two bounds as $|\overline{\calF} - \tilde{\calF}_{k}| = |T_1 + T_2| \leq \xi_{k} + \frac{\beta^{2}_{1} \xi_{k}}{1 + \xi_{k}}.$ This completes the proof. 

\section{Error analysis of Trace estimator}\label{ssec:trace}
 To recap the setup: Let $\vH$  and $\vQ$ be two symmetric positive semidefinite $n\times n$ matrices, with $\vQ$ positive definite. Let $\vomega_k \in \mathbb{R}^n$ for $1 \leq k \leq N$ be random vectors with independent entries that have zero mean and $\max_j \| (\vomega_k)_j\|_{\psi_2} \leq K_\psi$. Define the trace estimator 
\[ \widehat{\trace} (\vH\vQ) \equiv \frac{1}{N} \sum_{k=1}^N \vomega_j^\top \vH\vQ \vomega_j. \]

\begin{proof}[Proof of Proposition~\ref{prop:trace}] 
By the Hanson-Wright inequality~\cite[Theorem 6.2.1]{vershynin2018high}, for a matrix $\vK $ and a vector $\vx$ with independent mean zero subgaussian entries with subgaussian norm, at most $K_\psi$
\begin{equation}\label{eqn:hansonwright} \mathbb{P}\left\{ \vx^\top\vK \vx - \mathbb{E}(\vx^\top \vK \vx)| \geq t \right\} \leq 2\exp\left( - C_{\rm HW} N \min\left\{ \frac{t^2}{K_\psi^4\|\vK\|_F^2}, \frac{t}{K_\psi^2\|\vK\|_2}  \right\} \right) . \end{equation}
Apply this result with $\vK = \frac1N \diag( \vH\vQ, \dots,  \vH\vQ)$ and $\vx = \begin{bmatrix} \vw_1^\top & \dots  & \vw_N^\top\end{bmatrix}^\top $. Note that $ \vx^\top \vK\vx = \widehat{\trace}(\vH\vQ)  $  and $$ \mathbb{E}(\vx^\top \vK \vx) = \trace(\vK) = \trace(\vH\vQ) = \trace(\vH_{\vQ})$$ by the cyclic property of trace,  $\|\vK\|_F =  \|\vH\vQ\|_F/\sqrt{N}$, and $\|\vK\|_2 = \|\vH\vQ\|_2/N$. Plug these identities into~\eqref{eqn:hansonwright} to obtain~\eqref{eqn:probtrace}. 

For the second result, the tail bound in~\eqref{eqn:probtrace} can be bounded as 
\[ 2\exp\left( - C_{\rm HW} N \min\left\{ \frac{t^2}{K^4\|\vK\|_F^2}, \frac{t}{K_\psi^2\|\vK\|_2}  \right\} \right)  \leq 2\exp\left( - C_{\rm HW} N \frac{t^2}{ t \|\vH\vQ\|_2 +  K_\psi^2 \|\vH\vQ\|_F^2}   \right)  .\]
where we have used the inequality $\min\{x,y\} \geq xy/(x+y)$ for $x, y \geq 0$. 
Set $t = \epsilon \trace(\vH\vQ)$, the upper bound to $\delta $ and solve for $N$. Finally, note that $\trace(\vH\vQ) = \trace(\vH_{\vQ})$ is nonnegative.
\end{proof}

\bibliography{references}
\bibliographystyle{abbrv}

\end{document}